\documentclass[12pt]{amsart}       

\usepackage{txfonts}
\usepackage{amssymb}
\usepackage{eucal}
\usepackage{graphicx}
\usepackage{graphicx, float, url}
\usepackage{amsmath}
\usepackage{amscd}
\usepackage[all]{xy}           
\usepackage{latexsym}
\usepackage{mathrsfs}
\usepackage{amsthm}
\usepackage{txfonts}
\usepackage{tikz}
\usepackage{amsfonts,latexsym}
\usepackage{xspace}
\usepackage{epsfig}
\usepackage{float}
\usepackage{color}
\usepackage{fancybox}
\usepackage{colordvi}
\usepackage{multicol}
\usepackage{colordvi}
\usepackage{enumerate}
\usepackage[active]{srcltx} 
\usepackage{tikz}
\usepackage{url}

\newtheorem{theorem}{Theorem}[section]
\newtheorem{prop}[theorem]{Proposition}
\newtheorem{lemma}[theorem]{Lemma}

\theoremstyle{definition}
\newtheorem{defn}[theorem]{Definition}
\newtheorem{prop-def}{Proposition-Definition}[section]
\newtheorem{coro-def}{Corollary-Definition}[section]

\newtheorem{remark}[theorem]{Remark}

\newtheorem{exam}[theorem]{Example}

\newtheorem{conj}[theorem]{Conjecture}

\newcommand{\nc}{\newcommand}
\nc{\tred}[1]{\textcolor{red}{#1}}
\nc{\tblue}[1]{\textcolor{blue}{#1}}
\nc{\tgreen}[1]{\textcolor{green}{#1}}
\nc{\tpurple}[1]{\textcolor{purple}{#1}}
\nc{\btred}[1]{\textcolor{red}{\bf #1}}
\nc{\btblue}[1]{\textcolor{blue}{\bf #1}}
\nc{\btgreen}[1]{\textcolor{green}{\bf #1}}
\nc{\btpurple}[1]{\textcolor{purple}{\bf #1}}
\nc{\NN}{{\mathbb N}}
\nc{\ncsha}{{\mbox{\cyr X}^{\mathrm NC}}} \nc{\ncshao}{{\mbox{\cyr
X}^{\mathrm NC}_0}}

\newcommand{\efootnote}[1]{}

\renewcommand{\textbf}[1]{}

\newcommand{\delete}[1]{}

\nc{\mlabel}[1]{\label{#1}}  
\nc{\mcite}[1]{\cite{#1}}  
\nc{\mref}[1]{\ref{#1}}  
\nc{\mbibitem}[1]{\bibitem{#1}} 

\delete{
\nc{\mlabel}[1]{\label{#1}{\hfill \hspace{1cm}{\bf{{\ }\hfill(#1)}}}}
\nc{\mcite}[1]{\cite{#1}{{\bf{{\ }(#1)}}}}  
\nc{\mref}[1]{\ref{#1}{{\bf{{\ }(#1)}}}}  
\nc{\mbibitem}[1]{\bibitem[\bf #1]{#1}} 
}

\nc{\opa}{\ast} \nc{\opb}{\odot} \nc{\op}{\bullet} \nc{\pa}{\frakL}
\nc{\arr}{\rightarrow} \nc{\lu}[1]{(#1)} \nc{\mult}{\mrm{mult}}
\nc{\diff}{\mathfrak{Diff}}
\nc{\opc}{\sharp}\nc{\opd}{\natural}
\nc{\ope}{\circ}
\nc{\dpt}{\mathrm{d}}
\nc{\hck}{H_{RT}}
\nc{\vdf}{\calf}
\nc{\ldf}{\calf_\ell}
\nc{\hlf}{H_\ell}
\nc{\onek}{\mathbf{1}_\bfk}
\nc{\diam}{alternating\xspace}
\nc{\Diam}{Alternating\xspace}
\nc{\cdiam}{canonical alternating\xspace}
\nc{\Cdiam}{Canonical alternating\xspace}
\nc{\AW}{\mathcal{A}}

\nc{\ari}{\mathrm{ar}}

\nc{\lef}{\mathrm{lef}}

\nc{\Sh}{\mathrm{ST}}

\nc{\Cr}{\mathrm{Cr}}

\nc{\st}{{Schr\"oder tree}\xspace}
\nc{\sts}{{Schr\"oder trees}\xspace}

\nc{\vertset}{\Omega} 

\nc{\assop}{\quad \begin{picture}(5,5)(0,0)
\line(-1,1){10}
\put(-2.2,-2.2){$\bullet$}
\line(0,-1){10}\line(1,1){10}
\end{picture} \quad \smallskip}

\nc{\operator}{\begin{picture}(5,5)(0,0)
\line(0,-1){6}
\put(-2.6,-1.8){$\bullet$}
\line(0,1){9}
\end{picture}}

\nc{\idx}{\begin{picture}(6,6)(-3,-3)
\put(0,0){\line(0,1){6}}
\put(0,0){\line(0,-1){6}}
\end{picture}}

\nc{\pb}{{\mathrm{pb}}}
\nc{\Lf}{{\mathrm{Lf}}}

\nc{\lft}{{left tree}\xspace}
\nc{\lfts}{{left trees}\xspace}

\nc{\fat}{{fundamental averaging tree}\xspace}

\nc{\fats}{{fundamental averaging trees}\xspace}
\nc{\avt}{\mathrm{Avt}}

\nc{\rass}{{\mathit{RAss}}}

\nc{\aass}{{\mathit{AAss}}}

\nc{\vin}{{\mathrm Vin}}    
\nc{\lin}{{\mathrm Lin}}    
\nc{\inv}{\mathrm{I}n}
\nc{\gensp}{V} 
\nc{\genbas}{\mathcal{V}} 
\nc{\bvp}{V_P}     
\nc{\gop}{{\,\omega\,}}     

\nc{\bin}[2]{ (_{\stackrel{\scs{#1}}{\scs{#2}}})}  
\nc{\binc}[2]{ \left (\!\! \begin{array}{c} \scs{#1}\\
    \scs{#2} \end{array}\!\! \right )}  
\nc{\bincc}[2]{  \left ( {\scs{#1} \atop
    \vspace{-1cm}\scs{#2}} \right )}  
\nc{\bs}{\bar{S}} \nc{\cosum}{\sqsubset} \nc{\la}{\longrightarrow}
\nc{\rar}{\rightarrow} \nc{\dar}{\downarrow} \nc{\dprod}{**}
\nc{\dap}[1]{\downarrow \rlap{$\scriptstyle{#1}$}}
\nc{\md}{\mathrm{dth}} \nc{\uap}[1]{\uparrow
\rlap{$\scriptstyle{#1}$}} \nc{\defeq}{\stackrel{\rm def}{=}}
\nc{\disp}[1]{\displaystyle{#1}} \nc{\dotcup}{\
\displaystyle{\bigcup^\bullet}\ } \nc{\gzeta}{\bar{\zeta}}
\nc{\hcm}{\ \hat{,}\ } \nc{\hts}{\hat{\otimes}}
\nc{\barot}{{\otimes}} \nc{\free}[1]{\bar{#1}}
\nc{\uni}[1]{\tilde{#1}} \nc{\hcirc}{\hat{\circ}} \nc{\lleft}{[}
\nc{\lright}{]} \nc{\lc}{\lfloor} \nc{\rc}{\rfloor}
\nc{\curlyl}{\left \{ \begin{array}{c} {} \\ {} \end{array}
    \right .  \!\!\!\!\!\!\!}
\nc{\curlyr}{ \!\!\!\!\!\!\!
    \left . \begin{array}{c} {} \\ {} \end{array}
    \right \} }
\nc{\longmid}{\left | \begin{array}{c} {} \\ {} \end{array}
    \right . \!\!\!\!\!\!\!}
\nc{\onetree}{\bullet} \nc{\ora}[1]{\stackrel{#1}{\rar}}
\nc{\ola}[1]{\stackrel{#1}{\la}}
\nc{\ot}{\otimes} \nc{\mot}{{{\boxtimes\,}}}
\nc{\otm}{\overline{\boxtimes}} \nc{\sprod}{\bullet}
\nc{\scs}[1]{\scriptstyle{#1}} \nc{\mrm}[1]{{\rm #1}}
\nc{\margin}[1]{\marginpar{\rm #1}}   
\nc{\dirlim}{\displaystyle{\lim_{\longrightarrow}}\,}
\nc{\invlim}{\displaystyle{\lim_{\longleftarrow}}\,}
\nc{\mvp}{\vspace{0.3cm}} \nc{\tk}{^{(k)}} \nc{\tp}{^\prime}
\nc{\ttp}{^{\prime\prime}} \nc{\svp}{\vspace{2cm}}
\nc{\vp}{\vspace{8cm}} \nc{\proofbegin}{\noindent{\bf Proof: }}
\nc{\proofend}{$\blacksquare$ \vspace{0.3cm}}
\nc{\modg}[1]{\!<\!\!{#1}\!\!>}
\nc{\intg}[1]{F_C(#1)} \nc{\lmodg}{\!
<\!\!} \nc{\rmodg}{\!\!>\!}
\nc{\cpi}{\widehat{\Pi}}
\nc{\sha}{{\mbox{\cyr X}}}  
\nc{\shap}{{\mbox{\cyrs X}}} 
\nc{\shpr}{\diamond}    
\nc{\shp}{\ast} \nc{\shplus}{\shpr^+}
\nc{\shprc}{\shpr_c}    
\nc{\msh}{\ast} \nc{\zprod}{m_0} \nc{\oprod}{m_1}
\nc{\vep}{\varepsilon} \nc{\labs}{\mid\!} \nc{\rabs}{\!\mid}
\nc{\sqmon}[1]{\langle #1\rangle}

\nc{\mmbox}[1]{\mbox{\ #1\ }} \nc{\dep}{\mrm{dep}} \nc{\fp}{\mrm{FP}}
\nc{\rchar}{\mrm{char}} \nc{\End}{\mrm{End}} \nc{\Fil}{\mrm{Fil}}
\nc{\Mor}{Mor\xspace} \nc{\gmzvs}{gMZV\xspace}
\nc{\gmzv}{gMZV\xspace} \nc{\mzv}{MZV\xspace}
\nc{\mzvs}{MZVs\xspace} \nc{\Hom}{\mrm{Hom}} \nc{\id}{\mrm{id}}
\nc{\im}{\mrm{im}} \nc{\incl}{\mrm{incl}} \nc{\map}{\mrm{Map}}
\nc{\mchar}{\rm char} \nc{\nz}{\rm NZ} \nc{\supp}{\mathrm Supp}

\nc{\Alg}{\mathbf{Alg}} \nc{\Bax}{\mathbf{Bax}} \nc{\bff}{\mathbf f}
\nc{\bfk}{{\bf k}} \nc{\bfone}{{\bf 1}}
\nc{\bfx}{{\bf x}}
\nc{\bfy}{\mathbf y}
\nc{\base}[1]{\bfone^{\otimes ({#1}+1)}} 
\nc{\Cat}{\mathbf{Cat}}

\nc{\detail}{\marginpar{\bf More detail}
    \noindent{\bf Need more detail!}
    \svp}
\nc{\Int}{\mathbf{Int}} \nc{\Mon}{\mathbf{Mon}}
\nc{\rbtm}{{shuffle }} \nc{\rbto}{{Rota-Baxter }}
\nc{\remarks}{\noindent{\bf Remarks: }} \nc{\Rings}{\mathbf{Rings}}
\nc{\Sets}{\mathbf{Sets}} \nc{\wtot}{\widetilde{\odot}}
\nc{\wast}{\widetilde{\ast}} \nc{\bodot}{\bar{\odot}}
\nc{\bast}{\bar{\ast}} \nc{\hodot}[1]{\odot^{#1}}
\nc{\hast}[1]{\ast^{#1}} \nc{\mal}{\mathcal{O}}
\nc{\tet}{\tilde{\ast}} \nc{\teot}{\tilde{\odot}}
\nc{\oex}{\overline{x}} \nc{\oey}{\overline{y}}
\nc{\oez}{\overline{z}} \nc{\oef}{\overline{f}}
\nc{\oea}{\overline{a}} \nc{\oeb}{\overline{b}}
\nc{\weast}[1]{\widetilde{\ast}^{#1}}
\nc{\weodot}[1]{\widetilde{\odot}^{#1}} \nc{\hstar}[1]{\star^{#1}}
\nc{\lae}{\langle} \nc{\rae}{\rangle}
\nc{\lf}{\lfloor}
\nc{\rf}{\rfloor}

\nc{\QQ}{{\mathbb Q}}
\nc{\RR}{{\mathbb R}}
\nc{\ZZ}{{\mathbb Z}}

\nc{\cala}{{\mathcal A}} \nc{\calb}{{\mathcal B}}
\nc{\calc}{{\mathcal C}}
\nc{\cald}{{\mathcal D}} \nc{\cale}{{\mathcal E}}
\nc{\calf}{{\mathcal F}} \nc{\calg}{{\mathcal G}}
\nc{\calh}{{\mathcal H}} \nc{\cali}{{\mathcal I}}
\nc{\call}{{\mathcal L}} \nc{\calm}{{\mathcal M}}
\nc{\caln}{{\mathcal N}} \nc{\calo}{{\mathcal O}}
\nc{\calp}{{\mathcal P}} \nc{\calr}{{\mathcal R}}
\nc{\cals}{{\mathcal S}} \nc{\calt}{{\mathcal T}}
\nc{\calu}{{\mathcal U}} \nc{\calw}{{\mathcal W}} \nc{\calk}{{\mathcal K}}
\nc{\calx}{{\mathcal X}} \nc{\CA}{\mathcal{A}}

\nc{\fraka}{{\mathfrak a}} \nc{\frakA}{{\mathfrak A}}
\nc{\frakb}{{\mathfrak b}} \nc{\frakB}{{\mathfrak B}}
\nc{\frakD}{{\mathfrak D}} \nc{\frakF}{\mathfrak{F}}
\nc{\frakf}{{\mathfrak f}} \nc{\frakg}{{\mathfrak g}}
\nc{\frakH}{{\mathfrak H}} \nc{\frakL}{{\mathfrak L}}
\nc{\frakM}{{\mathfrak M}} \nc{\bfrakM}{\overline{\frakM}}
\nc{\frakm}{{\mathfrak m}} \nc{\frakP}{{\mathfrak P}}
\nc{\frakN}{{\mathfrak N}} \nc{\frakp}{{\mathfrak p}}
\nc{\frakS}{{\mathfrak S}} \nc{\frakT}{\mathfrak{T}}
\nc{\frakX}{{\mathfrak X}}
\nc{\BS}{\mathbb{S}}

\font\cyr=wncyr10 \font\cyrs=wncyr7
\nc{\li}[1]{\textcolor{red}{Li: #1}}
\nc{\xm}[1]{\textcolor{blue}{[xiaomeng: #1]}}
\nc{\zhou}[1]{\textcolor{purple}{[Sanming: #1]}}
\nc{\revise}[1]{\textcolor{red}{#1}}
\usepackage{xcolor}
\usepackage[normalem]{ulem}

\nc{\ID}{{\rm I}}
\nc{\lbar}[1]{\overline{#1}}
\nc{\bre}{{\rm bre}}
\nc{\sd}{\cals}\nc{\rb}{\rm RB}
\nc{\A}{\rm A}\nc{\LL}{\rm L}
\nc{\tx}{\tilde{X}}
\nc{\col}{\Delta_{\varepsilon}}
\nc{\mul}{m_{RT}}\nc{\ul}{u_{RT}}
\nc{\epl}{\varepsilon_{RT}}
\nc{\hl}{H_{RT}}\nc{\arro}[1]{#1}
\nc{\px}{P_{\tx}}
\nc{\pw}{P_{\mathfrak{w}}}
\nc{\pl}{B^+}
\nc{\Cay}{\mathrm{Cay}}
\nc{\caygs}{\mathrm{Cay}(G,S)}
\nc{\pp}{\pl}\nc{\ppp}[1]{B^+(#1)}
\nc{\cay}{\mathrm{Cay}(\mathbb{Z}_m,S)}
\nc{\dn}{D_{2n}}
\nc{\auttf}{\mathrm{Aut}_{K_2}}
\nc{\autgs}{\mathrm{Aut}(\Gamma\times \Sigma)}
\nc{\pge}{\mathrm{P}(\Gamma, C_{2k})}
\nc{\autge}{\mathrm{Aut}(\Gamma\times C_{2k})}
\nc{\pgo}{\mathrm{P}(\Gamma, C_{2k+1})}
\nc{\autgo}{\mathrm{Aut}(\Gamma\times C_{2k+1})}
\nc{\ncy}{N_{C_n}}
\nc{\nga}{N_{\Gamma}}
\nc{\ngc}{N_{\Gamma\times C_n}}
\nc{\autgc}{\mathrm{Aut}(\Gamma\times C_n)}
\nc{\pgs}{\mathrm{P}(\Gamma, \Sigma)}
\nc{\pgc}{\mathrm{P}(\Gamma, C_n)}
\nc{\fix}{{\mathrm Fix\,}}
\nc{\auts}{\mathrm{Aut}_{\Sigma}\,}
\nc{\aut}{\mathrm{Aut}\,}
\nc{\ben}{\beta_1, \ldots, \beta_n}
\nc{\aln}{\alpha_1, \ldots, \alpha_n}
\nc{\bfc}{{\mathrm B}(C_n)}

\begin{document}

\openup 0.5\jot

\title[stability of graph pairs]{stability of graph pairs involving cycles}
%
\author{Xiaomeng Wang}
\address{Department of Mathematics, Lanzhou University, Lanzhou, Gansu 730000, P.\,R. China}
         \email{wangxm2015@lzu.edu.cn}

\author{Shou-jun Xu}
\address{Department of Mathematics, Lanzhou University, Lanzhou, Gansu 730000, P.\,R. China}
         \email{shjxu@lzu.edu.cn}

\author{Sanming Zhou}
\address{School of Mathematics and Statistics, The University of Melbourne, Parkville, VIC 3010, Australia}
        \email{sanming@unimelb.edu.au}

\date{\today}
\begin{abstract}
A graph pair $(\Gamma, \Sigma)$ is called stable if $\aut(\Gamma)\times\aut(\Sigma)$ is isomorphic to $\aut(\Gamma\times\Sigma)$ and unstable otherwise, where $\Gamma\times\Sigma$ is the direct product of $\Gamma$ and $\Sigma$. A graph is called $R$-thin if distinct vertices have different neighbourhoods. $\Gamma$ and $\Sigma$ are said to be coprime if there is no nontrivial graph $\Delta$ such that $\Gamma \cong \Gamma_1 \times \Delta$ and $\Sigma \cong \Sigma_1 \times \Delta$ for some graphs $\Gamma_1$ and $\Sigma_1$. An unstable graph pair $(\Gamma, \Sigma)$ is called nontrivially unstable if $\Gamma$ and $\Sigma$ are $R$-thin connected coprime graphs and at least one of them is non-bipartite. This paper contributes to the study of the stability of graph pairs with a focus on the case when $\Sigma = C_n$ is a cycle. We introduce two key concepts in our study, namely the compatibility of $n$ with $\Gamma$ and an auxiliary graph $\Gamma^*$ on the same vertex set as $\Gamma$. We prove that for an $R$-thin connected graph $\Gamma$ and an integer $n \ge 3$ with $n \neq 4$ such that at least one of $\Gamma$ and $C_n$ is non-bipartite, if $n$ is compatible with $\Gamma$, or $n \ge 5$ is odd and for every edge $\{u, v\}$ of $\Gamma^*$ the set of common neighbours of $u$ and $v$ in $\Gamma$ is not an independent set of $\Gamma$, then $(\Gamma, C_n)$ is nontrivially unstable if and only if at least one $C_n$-automorphism of $\Gamma$ is nondiagonal. In the case when $\Gamma$ is an $R$-thin connected non-bipartite graph, we obtain the following results: (i) $(\Gamma, K_2)$ is unstable if and only if $(\Gamma, C_{n})$ is unstable for every even integer $n \geq 4$; (ii) if an even integer $n \ge 6$ is compatible with $\Gamma$, then $(\Gamma, C_{n})$ is nontrivially unstable if and only if $(\Gamma, K_2)$ is unstable; (iii) if there is an even integer $n \ge 6$ compatible with $\Gamma$ such that $(\Gamma, C_{n})$ is nontrivially unstable, then $(\Gamma, C_{m})$ is unstable for all even integers $m \ge 6$. We also prove that for an $R$-thin connected graph $\Gamma$ and an odd integer $n \ge 3$, if $n$ is compatible with $\Gamma$, or $n \ge 5$ and for every edge $\{u, v\}$ of $\Gamma^*$ the set of common neighbours of $u$ and $v$ in $\Gamma$ is not an independent set of $\Gamma$, then $(\Gamma, C_{n})$ is stable. Three conjectures arisen from our study are proposed in this paper.
\end{abstract}

\keywords{stable graph pair, stable graph, expected automorphism, direct product of graphs.}
\maketitle


\setcounter{section}{0}

\allowdisplaybreaks

\section{Introduction}

All graphs considered in this paper are finite and undirected with no loops or parallel edges. As usual, for a graph $\Gamma$, we use $V(\Gamma)$ and $E(\Gamma)$ to denote its vertex set and edge set respectively, and call $|V(\Gamma)|$ the order of $\Gamma$. The edge between two adjacent vertices $u, v$ of $\Gamma$ is denoted by $\{u, v\}$, the neighbourhood of a vertex $u$ in $\Gamma$ is denoted by $N_{\Gamma}(u)$, and the degree of $u$ in $\Gamma$ is defined as $\deg(u) = |N_{\Gamma}(u)|$. Denote by $d(u,v)$ the distance in $\Gamma$ between two vertices $u, v$ of $\Gamma$ and by $\Gamma[S]$ the subgraph of $\Gamma$ induced by a subset $S$ of $V(\Gamma)$. An {\em automorphism} of $\Gamma$ is a permutation $\sigma$ of $V(\Gamma)$ such that for any $u, v \in V(\Gamma)$, $\{u,v\}$ is an edge of $\Gamma$ if and only if $\{u^\sigma,v^\sigma\}$ is an edge of $\Gamma$, where for each $w \in V(\Gamma)$, $w^{\sigma}$ is the image of $w$ under $\sigma$. The \emph{automorphism group} of $\Gamma$, denoted by $\aut(\Gamma)$, is the group of automorphisms of $\Gamma$ under the composition of permutations. A graph $\Gamma$ is called \emph{vertex-transitive} if for any $u, v \in V(\Gamma)$ there exists an element $\sigma \in \aut(\Gamma)$ such that $u^{\sigma} = v$. We use $K_n$ to denote the complete graph with order $n \ge 1$ and $C_n$ the cycle with order $n \ge 3$. A graph is {\em trivial} if it has only one vertex and {\em nontrivial} otherwise.

The {\em direct product} \cite{ham11} of two graphs $\Gamma$ and $\Sigma$, denoted by $\Gamma \times \Sigma$, is the graph with vertex set $V(\Gamma) \times V(\Sigma)$ in which $(u, v)$ and $(x, y)$ are adjacent if and only if $u$ is adjacent to $x$ in $\Gamma$ and $v$ is adjacent to $y$ in $\Sigma$. In particular, $\Gamma \times K_2$ is the {\em canonical double cover} $D(\Gamma)$ of $\Gamma$. A graph is {\em prime} with respect to the direct product if it is nontrivial and cannot be represented as the direct product of two nontrivial graphs. Two graphs $\Gamma$ and $\Sigma$ are said to be \emph{coprime} with respect to the direct product if there is no nontrivial graph $\Delta$ such that $\Gamma \cong \Gamma' \times \Delta$ and $\Sigma \cong \Sigma' \times \Delta$ for some graphs $\Gamma'$ and $\Sigma'$. It is readily seen that for any graphs $\Gamma$ and $\Sigma$, $\aut(\Gamma)\times\aut(\Sigma)$ (direct product of groups) is isomorphic to a subgroup of $\aut(\Gamma\times\Sigma)$. As defined by Qin \textit{et al.} in \cite{qin211}, a graph pair $(\Gamma, \Sigma)$ is called {\em stable} if $\aut(\Gamma)\times\aut(\Sigma)$ is isomorphic to $\aut(\Gamma\times\Sigma)$ and {\em unstable} otherwise. This definition generalizes the notion of the stability of a graph \cite{maru89, wilson08} in the sense that $(\Gamma, K_2)$ is stable or unstable if and only if $\Gamma$ is stable or unstable, respectively. The stability of graphs has been studied considerably in the past more than three decades \cite{lauri15, maru89, maru92, ned96, sur01, sur03, wilson08}. For example, in \cite[Theorems C.1-C.4]{wilson08}, Wilson gave four sufficient conditions (see \cite{qin19} for an amendment to one of them) for a circulant graph to be unstable. In \cite[Conjecture 1.3]{qin19},  Qin \textit{et al.} conjectured that there is no nontrivially unstable circulant of odd order, and in \cite[Theorem 1.4]{qin19} they proved that this is true for circulants of prime order. The stability of the generalized Petersen graphs has been completely determined owing to  \cite[Theorems P.1-P.2]{wilson08} and \cite[Corollary 1.3]{qin21}. In contrast, the study of the stability of general graph pairs started only recently \cite{qin211, qin22}.

A graph $\Gamma$ is called {\em $R$-thin} (or {\em vertex-determining}) if $N_{\Gamma}(u) \ne N_{\Gamma}(v)$ for any two distinct vertices $u, v$ of $\Gamma$. Graphs which are not $R$-thin are called \emph{$R$-thick}. For example, $C_4$ is R-thick while $C_n$ is R-thin for any $n \ge 3$ with $n \neq 4$. An unstable graph is called \emph{nontrivially unstable} \cite{qin211, wilson08} if it is connected, non-bipartite and $R$-thin, and trivially unstable otherwise. In~\cite[Theorem 1.3]{qin211}, Qin \textit{et al.} proved that if $(\Gamma,\Sigma)$ is stable, then $\Gamma$ and $\Sigma$ are coprime $R$-thin graphs, and if in addition both $\aut(\Gamma)$ and $\aut(\Sigma)$ are nontrivial groups, then both $\Gamma$ and $\Sigma$ are connected and at least one of them is non-bipartite. Due to this result an unstable graph pair $(\Gamma,\Sigma)$ is called {\em nontrivially unstable} \cite{qin211} if $\Gamma$ and $\Sigma$ are $R$-thin connected coprime graphs and at least one of them is non-bipartite. This definition agrees with the concept of a nontrivially unstable graph introduced in \cite{wilson08} in the sense that $(\Gamma, K_2)$ is nontrivially unstable if and only if $\Gamma$ is nontrivially unstable.

Let $\Gamma$ and $\Sigma$ be graphs with $V(\Sigma)=\{1, \ldots, n\}$. As in \cite[Definition 2.3]{qin211}, we use $\pgs$ to denote the set of elements of $\aut(\Gamma\times\Sigma)$ that leave the partition $\{V(\Gamma)\times\{i\}:i \in V(\Sigma)\}$ invariant. Note that $\pgs$ is a subgroup of $\aut(\Gamma\times\Sigma)$. An $n$-tuple of permutations $(\aln)$ of $V(\Gamma)$ is called \cite{qin211} a {\em $\Sigma$-automorphism} of $\Gamma$ if for all $u, v \in V(\Gamma)$, $\{u, v\} \in E(\Gamma)$ if and only if $\{u^{\alpha_i}, v^{\alpha_j}\} \in E(\Gamma)$ for all $i, j \in V(\Sigma)$ with $\{i,j\} \in E(\Sigma)$. A $\Sigma$-automorphism $(\aln)$ of $\Gamma$ is said to be {\em nondiagonal} if $\alpha_i\neq \alpha_j$ for at least one pair of vertices $i, j$ of $\Sigma$. The set of all $\Sigma$-automorphisms of $\Gamma$ with operation defined by
$(\alpha_1, \ldots, \alpha_n)(\beta_1, \ldots, \beta_n) = (\alpha_1\beta_1, \ldots, \alpha_n\beta_n)$ is a group, written $\auts(\Gamma)$, which is called \cite{qin211} the \emph{$\Sigma$-automorphism group} of $\Gamma$. In particular, $\mathrm{Aut}_{K_2}(\Gamma)$ is exactly the \emph{two-fold automorphism group} of $\Gamma$ (see \cite{lauri04,lauri11,lauri15}) and its elements are called the \emph{two-fold automorphisms} of $\Gamma$.

In \cite[Theorem 1.8]{qin211}, Qin \textit{et al.} proved that, for a connected regular graph $\Gamma$ and a connected vertex-transitive graph $\Sigma$ with coprime degrees, if both $\Gamma$ and $\Sigma$ are $R$-thin and at least one of them is non-bipartite, then $(\Gamma, \Sigma)$ is nontrivially unstable if and only if at least one $\Sigma$-automorphism of $\Gamma$ is nondiagonal. In the same paper they proposed to study the stability of graph pairs $(\Gamma, \Sigma)$ for various special families of graphs $\Gamma$ and/or various special families of graphs $\Sigma$, including the case when $\Sigma$ is $K_n$ or $C_n$. In \cite[Theorem 6]{qin22}, Qin \textit{et al.} proved that, if $\Gamma$ and $\Sigma$ are regular of coprime degrees and $\Sigma$ is vertex-transitive, then $(\Gamma, \Sigma)$ is nontrivially unstable if and only if $(\Gamma, K_2)$ is nontrivially unstable. In the same paper they also studied the stability of $(\Gamma, C_n)$ in the case when $\Gamma$ is regular (\cite[Proposition 17]{qin22}), and they asked under what conditions the pair $(\Gamma, C_n)$ is nontrivially unstable given that $\Gamma$ is stable and $n\ge 6$ is even (\cite[Question 18]{qin22}).

Motivated by the works above, in this paper we study the stability of $(\Gamma, C_n)$, where $n \ge 3$ and $\Gamma$ needs not be regular. We identify two conditions on $\Gamma$ and $n$ which are closely related to the stability of $(\Gamma, C_n)$. We present these conditions by way of the following definition.

\begin{defn}
\label{defn:comp}
Given a graph $\Gamma$, define $\Gamma^*$ to be the graph with vertex set $V(\Gamma)$ in which $u, v \in V(\Gamma)$ are adjacent if and only if $\deg(u)=\deg(v) = 2|N_\Gamma(u)\cap N_\Gamma(v)|$.

For any $u \in V(\Gamma)$, if $u$ is an isolated vertex of $\Gamma^*$, define $L(u)=\{0\}$; if $u$ has degree $1$ in $\Gamma^*$, define $L(u)=\{0,1\}$; if $u$ has degree at least $2$ in $\Gamma^*$, define $L(u)$ to be the (not necessarily nonempty) set of lengths of the cycles containing $u$ in $\Gamma^*$.

We say that an integer $n \ge 3$ is \emph{compatible} with $\Gamma$ if either (i) $n$ is odd and $n \notin L(u)$ for at least one vertex $u \in V(\Gamma)
$, or (ii) $n$ is even and $n/2 \notin L(u)$ for at least one vertex $u \in V(\Gamma)$. An integer $n \ge 3$ is \emph{incompatible} with $\Gamma$ if it is not compatible with $\Gamma$.
\end{defn}

The main results in this paper are as follows.

\begin{theorem}\label{th:thr3.7}
Let $\Gamma$ be an $R$-thin connected graph and $n \ge 3$ an integer with $n\neq 4$. Suppose that at least one of  $\Gamma$ and $C_n$ is non-bipartite. Then the following statements hold:
\begin{enumerate}
\item
\label{it:it3.7.1}
if $n$ is compatible with $\Gamma$, then $(\Gamma, C_n)$ is nontrivially unstable if and only if at least one $C_n$-automorphism of $\Gamma$ is nondiagonal;
\item
\label{it:it3.7.2}
if $n \ge 5$ is odd and for every $\{u,v\} \in E(\Gamma^\ast)$, $\nga(u)\cap\nga(v)$ is not an independent set of $\Gamma$, then $(\Gamma, C_n)$ is nontrivially unstable if and only if  at least one $C_n$-automorphism of $\Gamma$ is nondiagonal.
\label{it:it3.115}
\end{enumerate}
\end{theorem}

\begin{theorem}
\label{co:co3.3}
Let $\Gamma$ be an $R$-thin connected non-bipartite graph.
\begin{enumerate}
\item $(\Gamma, K_2)$ is unstable if and only if $(\Gamma, C_{2k})$ is unstable for every integer $k \geq 2$.
\item If $k \ge 3$ is an integer such that $2k$ is compatible with $\Gamma$, then $(\Gamma, C_{2k})$ is nontrivially unstable if and only if $(\Gamma, K_2)$ is unstable.
\item If there exists an integer $k \ge 3$ such that $2k$ is compatible with $\Gamma$ and $(\Gamma, C_{2k})$ is nontrivially unstable, then $(\Gamma, C_{2l})$ is unstable for every $l \ge 3$.
\end{enumerate}
\end{theorem}

\begin{theorem}\label{th:th3.4}
Let $\Gamma$ be an $R$-thin connected graph and $k \ge 1$ an integer. If $2k+1$ is compatible with $\Gamma$, then $(\Gamma, C_{2k+1})$ is stable.
\end{theorem}

\begin{theorem}
\label{thm:coines}
Let $\Gamma$ be an $R$-thin connected graph and $k\geq 2$ an integer. If for every $\{u,v\}\in E(\Gamma^*)$, $N_{\Gamma}(u)\cap N_{\Gamma}(v)$ is not an independent set of $\Gamma$, then $(\Gamma, C_{2k+1})$ is stable.
\end{theorem}

Theorem \ref{th:thr3.7} is in the same spirit as \cite[Theorem 1.8]{qin211} where $\Gamma$ is regular and $\Sigma$ is vertex-transitive. Here in Theorem \ref{th:thr3.7} we do not need $\Gamma$ to be regular, but we assume $\Sigma$ is a cycle among other things. As will be seen in the last section, the ``only if" part in part (b) of Theorem \ref{co:co3.3} is not true if $2k$ is incompatible with $\Gamma$, and the statement in Theorem \ref{th:th3.4} is not true if $2k+1$ is incompatible with $\Gamma$.

The rest of this paper is organized as follows. In Section~\ref{sec:sec2}, we will prove several preliminary results on $\Sigma$-automorphisms of $\Gamma$ for a general pair $(\Gamma, \Sigma)$, thereby enriching the theories of $\Sigma$-automorphisms developed in \cite{qin211}. We will also give in Section~\ref{sec:sec2} an example and a remark to explain the concepts introduced in Definition \ref{defn:comp}. The proof of Theorem \ref{th:thr3.7} consists of three lemmas which will be proved in Section~\ref{sec:sec3}. The proofs of Theorems \ref{co:co3.3}, \ref{th:th3.4} and \ref{thm:coines} will be given in Section \ref{sec:uaut} with the help of some known results from \cite{lauri15, wilson08} and a few technical lemmas produced in the present paper. The last section, Section \ref{sec:con-rem}, consists of some remarks on our main results and three conjectures in relation to Theorems \ref{th:thr3.7}, \ref{th:th3.4} and \ref{thm:coines}.

\section{Preliminaries}
\label{sec:sec2}

We begin this section with an example to illustrate Definition \ref{defn:comp} and a remark to give more explanations on the auxiliary graph $\Gamma^*$ and the compatibility of $n$ with $\Gamma$.

\begin{exam}
\label{ex:exin}
Consider the graph $\Gamma$ on the left-hand side of Figure \ref{fig:gammastar}. Clearly, for any $u,v \in \{u_1,u_3,u_5\}$ and any $u,v \in \{u_2,u_4,u_6\}$, we have $\deg(u) = \deg(v) = 2|N_\Gamma(u)\cap N_\Gamma(v)|$. Thus $\Gamma^*$ has two components each isomorphic to $C_3$. Hence $L(u) = \{3\}$ for all $u \in V(\Gamma)$. So $3$ and $6$ are incompatible with $\Gamma$, while every other integer $n \ge 3$ is compatible with $\Gamma$.
\end{exam}

\begin{figure}[ht]
\centering
\vspace{-0.2cm}
\includegraphics*[height=7.0cm]{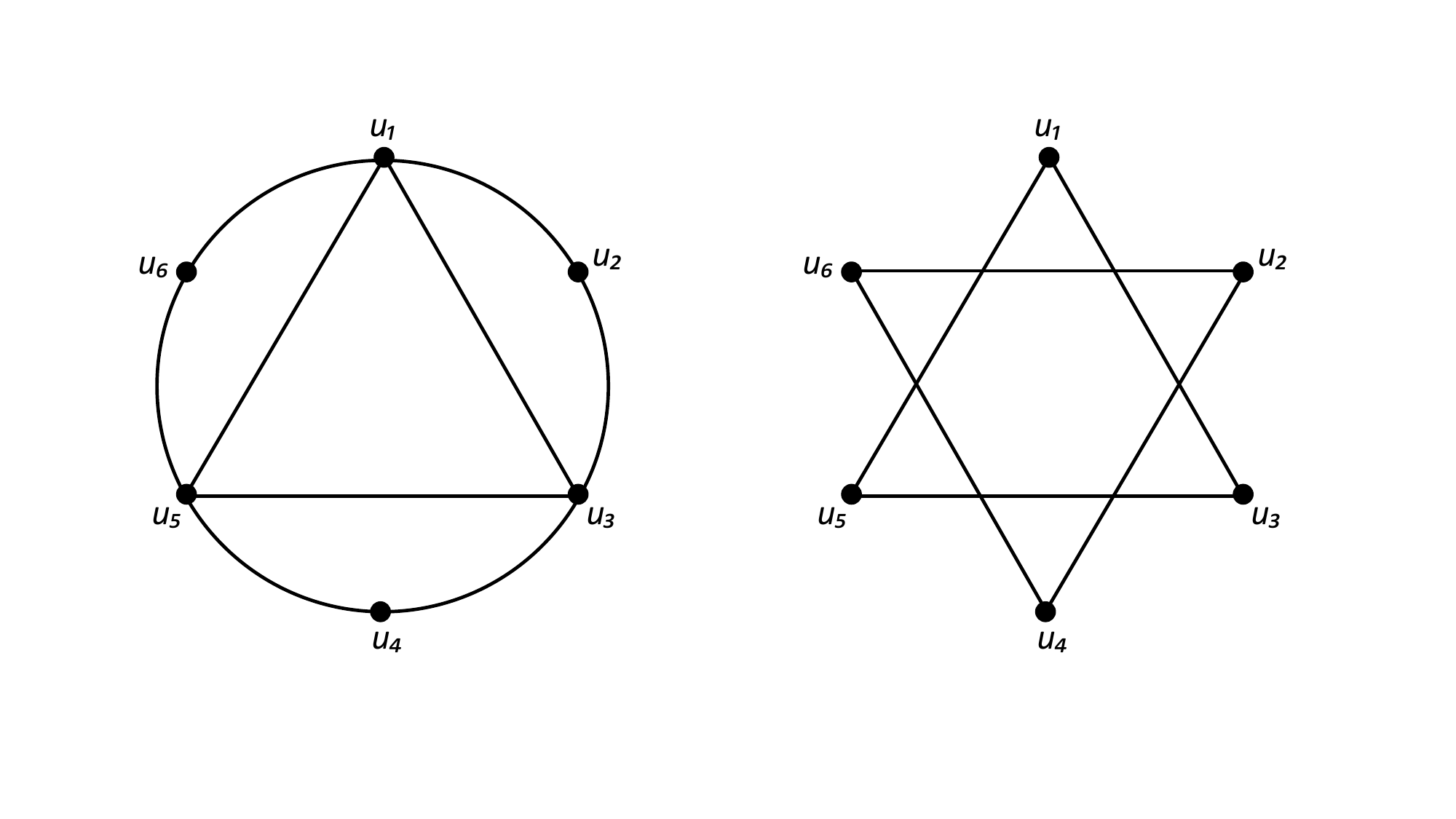}
\vspace{-1cm}
\caption{\small A graph $\Gamma$ (left) and its auxiliary graph $\Gamma^*$ (right).}
\label{fig:gammastar}
\end{figure}

\delete
{
\begin{center}
\begin{tikzpicture}
\coordinate[label=left:$u_{1}$] (a1) at (-2,0);
\coordinate[label=left:$u_{2}$] (a2) at (-1,2);
\coordinate[label=left:$u_{3}$] (a3) at (0,0);
\coordinate[label=right:$u_{4}$] (a4) at (1,2);
\coordinate[label=left:$u_{5}$] (a5) at (2,0);
\coordinate[label=right:$u_{6}$] (a6) at (0,-2);
\foreach \i in {a1, a3,a2,a4,a5,a6} {\draw (\i) circle (2pt);}
\draw (a1) -- (a2)--(a3)--(a4)--(a5)--(a6)--(a1);
\draw (0,0) arc [radius = 1.4, start angle = 45, end angle = 135];
\draw (2,0) arc [radius = 1.4, start angle = 45, end angle = 135];
\draw (2,0) arc [radius = 2.8, start angle = 315, end angle = 225];
\node [below] at (0,-2.2) {Fig. 1  The graph $\Gamma$};
\end{tikzpicture}
\end{center}
}

\begin{remark}
\label{rem:def}
(i) Note that adjacent vertices of $\Gamma^*$ may or may not be adjacent in $\Gamma$, and any two isolated vertices of $\Gamma$ are adjacent in $\Gamma^*$. Note also that adjacent vertices of $\Gamma^*$ have the same degree in $\Gamma$, and this degree is even.

(ii) If a vertex $u$ has degree at least $2$ in $\Gamma^*$ but is not contained in any cycle of $\Gamma^*$, then $L(u) = \emptyset$.

(iii) If $\Gamma^*$ has minimum degree at least $2$, then every vertex $u \in V(\Gamma)$ is contained in at least one cycle in $\Gamma^*$. Hence $L(u) \neq \emptyset$ and each member of $L(u)$  is an integer no less than $3$.

(iv) If there exists a vertex $u \in V(\Gamma)$ such that $L(u) = \emptyset$, then every integer $n \ge 3$ is compatible with $\Gamma$.

(v) If the minimum degree of $\Gamma^*$ is $0$ or $1$, then every integer $n \ge 3$ is compatible with $\Gamma$. In particular, if $\Gamma$ has a vertex $u$ of odd degree, then $u$ is an isolated vertex of $\Gamma^*$ and hence every integer $n \ge 3$ is compatible with $\Gamma$.

(vi) If an integer $n \ge 3$ is compatible with $\Gamma$, then either $\Gamma$ and $C_n$ are coprime, or $\Gamma$ is bipartite and $n \equiv 2\pmod{4}$. To prove this statement, we assume $n \ge 3$ is compatible with $\Gamma$ and $\Gamma$ and $C_n$ are not coprime. We aim to prove $\Gamma$ is bipartite and $n \equiv 2\pmod{4}$. Suppose to the contrary that $n$ is odd or $n=4k+4$ for some $k\geq 0$. Then $C_n$ is prime with respect to the direct product. Since $\Gamma$ and $C_n$ are not coprime, it follows that $\Gamma=\Sigma\times C_n$ for some nontrivial graph $\Sigma$. Set $V(C_n)=\{0,1,\ldots, n-1\}$. For any $u\in V(\Sigma)$ and $i\in V(C_n)$, we have $|N_{\Gamma}((u,i))\cap N_{\Gamma}((u,i+2))| = (1/2)\ |N_{\Gamma}((u,i))|$. So, if $n$ is odd, then $(u,0), (u,2), \ldots, (u, n-1), (u,1),\ldots, (u,n-2), (u, 0)$ is a cycle in $\Gamma^*$  with length $n$ and hence $n\in L((u,i))$ for any $(u,i)\in V(\Gamma)$. If $n=4k+4$ for some $k\geq 0$, then $(u,0), (u,2), \ldots, (u, n-4), (u,n-2), (u, 0)$ is a cycle in $\Gamma^*$ with length $n/2$ and hence $n/2\in L((u,i))$ for any $(u,i)\in V(\Gamma)$. In either case we obtain that $n$ is incompatible with $\Gamma$, but this contradicts our assumption. Thus, $n=4k+2$ for some $k\geq 1$, and hence $C_n=C_{2k+1}\times K_2$. If $\Gamma=\Sigma\times C_{2k+1}$  for some nontrivial graph $\Sigma$, then similarly to the case $n=4k+4$ above we obtain that $n$ is incompatible with $\Gamma$, a contradiction. Hence $\Gamma$ and $C_{2k+1}$ are coprime. Since $\Gamma$ and $C_n$ are not coprime, we must have $\Gamma=\Sigma\times K_2$ for some nontrivial graph $\Sigma$ and therefore $\Gamma$ is bipartite.
\end{remark}

Let $\Gamma$ and $\Sigma$ be graphs with $V(\Sigma)=\{1, \ldots, n\}$, and let $\phi \in \aut(\Sigma)$. Then
$$
(\alpha_1, \ldots,\alpha_n)^\phi = (\alpha_{1^\phi}, \ldots,\alpha_{n^\phi}),\; \text{for $(\aln)\in \auts(\Gamma)$}
$$
defines a bijection from $\auts(\Gamma)$ to $\auts(\Gamma)$.

\begin{lemma}
\label{le:le1.1}
Let $\Gamma$ and $\Sigma$ be connected graphs with $V(\Sigma)=\{1, \ldots, n\}$. Then for any $(\alpha_1, \ldots, \alpha_n)\in \auts(\Gamma)$ and $\phi\in\aut(\Sigma)$ we have $(\aln)^\phi \in \auts(\Gamma)$. Moreover, $\aut(\Sigma)$ is a subgroup of the automorphism group of $\auts(\Gamma)$.
\end{lemma}

\begin{proof}
Since $\phi\in\aut(\Sigma)$, for all $i', j' \in V(\Sigma)$, there exist $i, j \in V(\Sigma)$ such that $i' = i^{\phi}$ and $j' = j^{\phi}$. Moreover, $(i, j)$ and $(i', j')$ determine each other uniquely, and $\{i' ,j'\} \in E(\Sigma)$ if and only if $\{i, j\} \in E(\Sigma)$. Since $(\alpha_1, \ldots, \alpha_n)\in \auts(\Gamma)$, for any $u, v \in V(\Gamma)$, $\{u, v\} \in E(\Gamma)$ if and only if $\{u^{\alpha_{i'}}, v^{\alpha_{j'}}\} \in E(\Gamma)$ for any $i', j' \in V(\Sigma)$ with $\{i',j'\} \in E(\Sigma)$. That is, $\{u, v\} \in E(\Gamma)$ if and only if $\{u^{\alpha_{i^\phi}}, v^{\alpha_{j^\phi}}\} \in E(\Gamma)$ for any $i^\phi, j^\phi \in V(\Sigma)$ with $\{i^\phi ,j^\phi\} \in E(\Sigma)$. Therefore, $(\aln)^\phi \in \auts(\Gamma)$.

Now for any $\phi \in \aut(\Sigma)$ and $(\aln), (\ben) \in \auts(\Gamma)$, set $\gamma_i = \alpha_i\beta_i$ and $i' = i^\phi$ for $i \in V(\Sigma)$. Since $\phi \in \aut(\Sigma)$, when $i$ ranges over all vertices of $\Sigma$ so does $i'$. Hence
\begin{align*}
(\aln)^\phi (\beta_1, \ldots, \beta_n)^\phi &=\,(\alpha_{1^\phi}, \ldots,\alpha_{n^\phi})(\beta_{1^\phi}, \ldots, \beta_{n^\phi})\\
&=\,(\alpha_{1'}\beta_{1'}, \ldots,\alpha_{n'}\beta_{n'})\\
&=\,(\gamma_{1'}, \ldots, \gamma_{n'})\\
&=\,(\gamma_{1^\phi}, \ldots, \gamma_{n^\phi})\\
&=\,(\gamma_{1}, \ldots, \gamma_{n})^{\phi}\\
&=\,(\alpha_{1}\beta_{1}, \ldots,\alpha_{n}\beta_{n})^\phi\\
&=\,((\alpha_{1}, \ldots,\alpha_{n})(\beta_{1}, \ldots,\beta_{n}))^\phi.
\end{align*}
This means that $\aut(\Sigma)$ is a subgroup of the automorphism group of $\auts(\Gamma)$.
\end{proof}

In \cite[Proposition 3.2]{lauri11}, Lauri \textit{et al.} proved that for any two-fold automorphism $(\alpha_1, \alpha_2)$ of a graph $\Gamma$, either both $\alpha_1$ and $\alpha_2$ are automorphisms of $\Gamma$, or neither of them is. The following is a generalization of this result.

\begin{lemma}\label{le:le1.2}
Let $\Gamma$ and $\Sigma$ be connected graphs with $V(\Sigma)=\{1, \ldots, n\}$. Then for any $(\alpha_1, \ldots, \alpha_n)\in \auts(\Gamma)$, either $\alpha_i \in \aut(\Gamma)$ for all $1\leq i\leq n$ or $\alpha_i \notin \aut(\Gamma)$ for all $1\leq i\leq n$.
\end{lemma}

\begin{proof}
Suppose that $\alpha_i \in \aut(\Gamma)$ for some $i$. Then $(\alpha^{-1}_i, \ldots, \alpha^{-1}_i) \in \auts(\Gamma)$. Since $(\alpha_1, \ldots, \alpha_n)\in \auts(\Gamma)$ and $\auts(\Gamma)$ is a group, we have $(\alpha_1\alpha^{-1}_i,  \ldots, \alpha_i \alpha^{-1}_i, \ldots, \alpha_n\alpha^{-1}_i) \in \auts(\Gamma)$. Since $\alpha_i \alpha^{-1}_i = \id$, this shows that we may assume without loss of generality that $\alpha_i = \id$ in $(\alpha_1, \ldots, \alpha_n)\in \auts(\Gamma)$. Then for any $u, v \in V(\Gamma)$ and $j \in N_{\Sigma}(i)$, $\{u, v\} \in E(\Gamma)$ if and only if $\{u^{\alpha_i}, v^{\alpha_j}\} = \{u, v^{\alpha_j}\} \in E(\Gamma)$. Moreover, if $\{u,v^{\alpha_j}\} \in E(\Gamma)$, then $\{v^{\alpha_j}, u\} \in E(\Gamma)$ and hence $\{(v^{\alpha_j})^{\alpha_i}, u^{\alpha_j}\}=\{v^{\alpha_j}, u^{\alpha_j}\} \in E(\Gamma)$. So we have proved that $\{u,v\}\in E(\Gamma)$ implies $\{u^{\alpha_j}, v^{\alpha_j}\} \in E(\Gamma)$. In other words, $\alpha_j$ maps edges of $\Gamma$ to edges of $\Gamma$. Since $\alpha_j$ is a permutation of $V(\Gamma)$, we obtain further that $\alpha_j$ maps different edges of $\Gamma$ to different edges of $\Gamma$. So $\{u,v\}\notin E(\Gamma)$ implies $\{u^{\alpha_j}, v^{\alpha_j}\}\notin E(\Gamma)$. Therefore, $\alpha_j \in \aut(\Gamma)$ for all $j \in N_{\Sigma}(i)$.

If $N_{\Sigma}(j)\subseteq N_{\Sigma}(i)$ for any vertex $j\in N_{\Sigma}(i)$, then $N_{\Sigma}(i)\cup\{i\}=V(\Sigma)$ and so $\alpha_j\in \aut(\Gamma)$ for all $j\in N_{\Sigma}(i)$. We now consider the case $V(\Sigma) \setminus N_{\Sigma}(i)\neq \{i\}$. Define $(\ben)$ such that $\beta_j = \alpha_j$ for $j \in N_{\Sigma}(i)$ and $\beta_j = \id$ for $j \in V(\Sigma) \setminus N_{\Sigma}(i)$. Then for any $u, v \in V(\Gamma)$ and $\{k, l\} \in E(\Sigma)$, $\{u^{\beta_{k}}, v^{\beta_{l}}\}$ is equal to $\{u, v\}$, $\{u^{\alpha_{k}}, v\}$, $\{u, v^{\alpha_{l}}\}$ or $\{u^{\alpha_{k}}, v^{\alpha_{l}}\}$, depending on whether $k \in N_{\Sigma}(i)$ and/or $l \in N_{\Sigma}(i)$. This together with what we proved in the previous paragraph implies that $\{u, v\} \in E(\Gamma)$ if and only if $\{u^{\beta_{k}}, v^{\beta_{l}}\} \in E(\Gamma)$. Hence $(\ben) \in \auts(\Gamma)$. Since $\auts(\Gamma)$ is a group, we then have $(\beta_1^{-1}, \ldots, \beta_n^{-1}) \in \auts(\Gamma)$ and consequently $(\alpha_1\beta_1^{-1}, \ldots, \alpha_n\beta_n^{-1}) \in \auts(\Gamma)$. Note that $\alpha_j \beta^{-1}_{j}=\id$ for $j \in N_{\Sigma}(i)$ and $\alpha_j \beta^{-1}_{j}=\alpha_j$ for $j \in V(\Sigma) \setminus N_{\Sigma}(i)$. Thus, for any $\{u,v\}\in E(\Gamma)$ and any $j\in N_{\Sigma}(i)$ and $l\in N_{\Sigma}(j)$,  $\{u,v\}\in E(\Gamma)$ if and only if $\{u^{\alpha_j\beta_j^{-1}},v^{\alpha_l\beta_l^{-1}}\}=\{u,v^{\alpha_l}\} \in E(\Gamma)$. On the other hand, by $\{u, v^{\alpha_l}\}\in E(\Gamma)$ we have $\{v^{\alpha_l},u\}\in E(\Gamma)$ and so $\{(v^{\alpha_l})^{\alpha_j\beta_j^{-1}}, u^{\alpha_l\beta_l^{-1}}\}=\{v^{\alpha_l},u^{\alpha_l}\}\in E(\Gamma)$. Thus $\{u,v\}\in E(\Gamma)$ implies $\{u^{\alpha_l}, v^{\alpha_l}\}\in E(\Gamma)$ for all $l\in N_{\Sigma}(j)$ with $j\in N_{\Sigma}(i)$. That is, $\alpha_l$ maps edges of $\Gamma$ to edges of $\Gamma$. Since $\alpha_l$ is a permutation of $V(\Gamma)$, it must map different edges of $\Gamma$ to different edges of $\Gamma$. Thus $\{u, v\}\notin E(\Gamma)$ implies $\{u^{\alpha_l}, v^{\alpha_l}\}\notin E(\Gamma)$. Therefore, $\alpha_l\in \aut(\Gamma)$ for all $l \in N_{\Sigma}(j)$ with $j \in N_{\Sigma}(i)$. Since $\Sigma$ is a finite connected graph, by repeating the precess a finite number of times we obtain that $\alpha_j \in \aut(\Gamma)$ for all $1\leq j \leq n$.
\end{proof}

\delete
{
\begin{lemma}
\label{pro:pro2.4}
Let $\Gamma$ and $\Sigma$ be  connected graphs with $V(\Sigma)=\{1,\ldots,n\}$, and let $(\aln) \in  \auts(\Gamma)$ be a nondiagonal $\Sigma$-automorphism of $\Gamma$. If one of the following conditions holds, then $\Gamma$ is not $R$-thin:
\begin{enumerate}
\item
\label{pr:pr1.5}
$\alpha_i \in \aut(\Gamma)$ for at least one $1 \le i \le n$;
\item
\label{pr:pr1.6}
there exist $1 \le i, j \le n$ such that $\alpha_i$ and $\alpha_j$ have different orders;
\item
\label{le:le1.3}
there exists $\{i,j\} \in E(\Sigma)$ such that $\alpha_i=\alpha_j$.
\end{enumerate}
\end{lemma}

\begin{proof}
(\ref{pr:pr1.5})
Since $(\aln) \in \auts(\Gamma)$ and $\alpha_i \in \aut(\Gamma)$ for at least one $1 \le i \le n$, by Lemma~\ref{le:le1.2} we have $\alpha_i \in \aut(\Gamma)$ for all $1 \le i \le n$. Since $\Sigma$ is connected and $(\aln)$ is nondiagonal, there exists an edge $\{i,j\} \in E(\Sigma)$ such that $\alpha_i\neq\alpha_j$ and hence $\gamma := \alpha_j\alpha^{-1}_i$ is a non-identity element of $\aut(\Gamma)$. So $\gamma$ moves at least one vertex of $\Gamma$, say, $u$, so that $v := u^{\gamma} \ne u$. Assume without loss of generality that $i < j$. Then $(\alpha_1\alpha_i^{-1},\ldots, \alpha_i\alpha^{-1}_i, \ldots, \alpha_j\alpha_i^{-1},\ldots, \alpha_n\alpha_i^{-1}) \in \auts(\Gamma)$ by Lemma~\ref{le:le1.1}. For each $w \in N_\Gamma(u)$, we have $\{w, u^{\gamma}\}=\{w,v\}$. Hence each neighbour $w$ of $u$ in $\Gamma$ is also a neighbour of $v$ in $\Gamma$. So $N_\Gamma(u) \subseteq N_\Gamma(v)$. By Lemma \ref{le:le1.1} again, we have $((a_1a_i^{-1})^{-1},\ldots,(a_ia_i^{-1})^{-1},\ldots,(a_ja_i^{-1})^{-1},\ldots,(a_na_i^{-1})^{-1})\in\auts(\Gamma)$. Since $\gamma^{-1}=(\alpha_j\alpha_i^{-1})^{-1}$, we obtain that $\{w,v^{\gamma^{-1}}\}=\{w,u\} \in E(\Gamma)$ for any $w\in N_{\Gamma}(v)$. In other words, $N_\Gamma(v)\subseteq N_\Gamma(u)$. Therefore, $N_\Gamma(u) = N_\Gamma(v)$ and thus $\Gamma$ is not $R$-thin.

(\ref{pr:pr1.6})
Since $\Sigma$ is connected, the assumption that $\alpha_i$ and $\alpha_j$ have different orders for some $1 \le i, j \le n$ implies that there exist $1 \le i, j \le n$ such that $\{i, j\} \in E(\Sigma)$ and $\alpha_i$ and $\alpha_j$ have different orders. Without loss of generality we may assume $\{1,n\} \in E(\Sigma)$ and $\alpha_1, \alpha_n$ have orders $p, q$, respectively, for some $p < q$. Since $(\aln) \in  \auts(\Gamma)$, we have $(\alpha_1^p,\alpha_2^p,\ldots,\alpha^p_n) = (\id,\alpha_2^p,\ldots,\alpha^p_n)  \in \auts(\Gamma)$. Moreover, $(\id,\alpha_2^p,\ldots,\alpha^p_n)$ is nondiagonal as $\alpha^p_n \ne \id$. Since $\id \in \aut(\Gamma)$, it follows from part~(\ref{pr:pr1.5}) that $\Gamma$ is not $R$-thin.

(\ref{le:le1.3})
Assume $\alpha_i=\alpha_j$ for some $\{i,j\} \in E(\Sigma)$. Since $(\aln) \in \auts(\Gamma)$, for any $\{u,v\} \in E(\Gamma)$, we have $\{u^{\alpha_i}, v^{\alpha_i}\} = \{u^{\alpha_i}, v^{\alpha_j}\} \in E(\Gamma)$. Hence $\alpha_i \in \aut(\Gamma)$ and thus $\Gamma$ is not $R$-thin by part~(\ref{pr:pr1.5}).
\end{proof}

We can partition $\auts(\Gamma)$ into the following three parts:
\[\auts^{=}(\Gamma)=\{(\alpha,\ldots, \alpha)\in \auts(\Gamma):\alpha\in \aut(\Gamma)\}\]
\[\auts^{\text{NI}}(\Gamma)=\{(\aln) \in\auts(\Gamma): \alpha_i\notin \aut(\Gamma) \text{ for all } 1\leq i\leq n\}\]
\[\auts^{\text{I}}(\Gamma)=\{(\aln) \in \auts(\Gamma) \setminus \auts^{=}(\Gamma): \alpha_i\in \aut(\Gamma) \text{ for at least one } 1\leq i\leq n\}.\]
Lemma \ref{pro:pro2.4} implies the following result.
}

\begin{lemma}
\label{re:re1.1}
Let $\Gamma$ and $\Sigma$ be connected graphs with $V(\Sigma)=\{1,\ldots,n\}$. If $\Gamma$ is $R$-thin, then the following statements hold for every nondiagonal $(\aln) \in \auts(\Gamma)$:
\begin{enumerate}
\item
$\alpha_i \notin \aut(\Gamma)$ for $1 \le i \le n$;
\item
$\alpha_i$ for $1 \le i \le n$ all have the same order;
\item
$\alpha_i \neq \alpha_j$ for every $\{i,j\} \in E(\Sigma)$.
\end{enumerate}
\end{lemma}

\begin{proof}
Suppose that $\Gamma$ is $R$-thin. Let $(\aln) \in \auts(\Gamma)$ be an arbitrary nondiagonal $\Sigma$-automorphism of $\Gamma$.

(a) Suppose for a contradiction that $\alpha_i \in \aut(\Gamma)$ for at least one $1 \le i \le n$. Then by Lemma~\ref{le:le1.2} we have $\alpha_i \in \aut(\Gamma)$ for all $1 \le i \le n$. Since $\Sigma$ is connected and $(\aln)$ is nondiagonal, there exists an edge $\{i,j\} \in E(\Sigma)$ such that $\alpha_i\neq\alpha_j$ and hence $\gamma := \alpha_j\alpha^{-1}_i$ is a non-identity element of $\aut(\Gamma)$. So $\gamma$ moves at least one vertex of $\Gamma$, say, $u$, so that $v := u^{\gamma} \ne u$. Assume without loss of generality that $i < j$. Then $(\alpha_1\alpha_i^{-1},\ldots, \alpha_i\alpha^{-1}_i, \ldots, \alpha_j\alpha_i^{-1},\ldots, \alpha_n\alpha_i^{-1}) \in \auts(\Gamma)$ by Lemma~\ref{le:le1.1}. For each $w \in N_\Gamma(u)$, we have $\{w, u^{\gamma}\}=\{w,v\}$. Hence each neighbour $w$ of $u$ in $\Gamma$ is also a neighbour of $v$ in $\Gamma$. So $N_\Gamma(u) \subseteq N_\Gamma(v)$. By Lemma \ref{le:le1.1} again, we have $((\alpha_1 \alpha_i^{-1})^{-1},\ldots,(\alpha_i \alpha_i^{-1})^{-1},\ldots,(\alpha_j \alpha_i^{-1})^{-1},\ldots,(\alpha_n \alpha_i^{-1})^{-1})\in\auts(\Gamma)$. Since $\gamma^{-1}=(\alpha_j\alpha_i^{-1})^{-1}$, we obtain that $\{w,v^{\gamma^{-1}}\}=\{w,u\} \in E(\Gamma)$ for any $w\in N_{\Gamma}(v)$. In other words, $N_\Gamma(v)\subseteq N_\Gamma(u)$. Therefore, $N_\Gamma(u) = N_\Gamma(v)$ and thus $\Gamma$ is not $R$-thin, but this is a contradiction.

(b) Suppose to the contrary that there exist $1 \le i, j \le n$ such that $\alpha_i$ and $\alpha_j$ have different orders. Since $\Sigma$ is connected, this assumption implies that there exist $1 \le i, j \le n$ with $\{i, j\} \in E(\Sigma)$ such that $\alpha_i$ and $\alpha_j$ have different orders. Without loss of generality we may assume $\{1,n\} \in E(\Sigma)$ and $\alpha_1, \alpha_n$ have orders $p, q$, respectively, for some $p < q$. Since $(\aln) \in  \auts(\Gamma)$, we have $(\alpha_1^p,\alpha_2^p,\ldots,\alpha^p_n) = (\id,\alpha_2^p,\ldots,\alpha^p_n)  \in \auts(\Gamma)$. Moreover, $(\id,\alpha_2^p,\ldots,\alpha^p_n)$ is nondiagonal as $\alpha^p_n \ne \id$. Since $\id \in \aut(\Gamma)$, this contradicts what we have proved in part (a).

(c) Suppose that there exists some $\{i,j\} \in E(\Sigma)$ such that $\alpha_i=\alpha_j$. Since $(\aln) \in \auts(\Gamma)$ and $\{i,j\} \in E(\Sigma)$, for any $\{u,v\} \in E(\Gamma)$, we have $\{u^{\alpha_i}, v^{\alpha_i}\} = \{u^{\alpha_i}, v^{\alpha_j}\} \in E(\Gamma)$. Hence $\alpha_i \in \aut(\Gamma)$. But this contradicts what we have proved in part (a).
\end{proof}

\begin{prop}
\label{th:th1.1}
Let $\Gamma$ be an $R$-thin connected graph and $\Sigma$ a connected graph with $V(\Sigma) = \{1, \ldots, n\}$. Let $\Delta$ be a subgraph  of $\Sigma$. Then any $(\aln) \in \auts(\Gamma)$ gives rise to a $\Delta$-automorphism of $\Gamma$, namely $(\alpha_{i})_{i \in V(\Delta)} \in \aut_{\Delta}(\Gamma)$.
\end{prop}

\begin{proof}
Since $\Delta$ is a subgraph of $\Sigma$, without loss of generality we may assume $V(\Delta) = \{1,\ldots,m\}$. Then for $i, j \in V(\Delta)$, $\{i, j\} \in E(\Sigma)$ whenever $\{i,j\} \in E(\Delta)$. Since $(\aln) \in \auts(\Gamma)$, for all $u, v \in V(\Gamma)$ and any $\{i,j\} \in E(\Sigma)$, $\{u,v\} \in E(\Gamma)$ if and only if $\{u^{\alpha_i}, v^{\alpha_j}\} \in E(\Gamma)$. Since $E(\Delta) \subseteq E(\Sigma)$, the same statement holds for $\{i,j\} \in E(\Delta)$ and hence $(\alpha_{i})_{i \in V(\Delta)} \in \aut_{\Delta}(\Gamma)$.
\end{proof}

\delete
{
The next two lemmas will not be used in the proof of any result in this paper, but they may be useful elsewhere. Let $\Gamma_1, \Gamma_2$ and $\Sigma$ be graphs with $V(\Sigma) = \{1, \ldots, n\}$. An $n$-tuple of bijections $(\aln)$ from $V(\Gamma_1)$ to $V(\Gamma_2)$ is called a {\em $\Sigma$-isomorphism} from $\Gamma_1$ to $\Gamma_2$ if for all $u, v \in V(\Gamma_1)$, $\{u, v\} \in E(\Gamma_1)$ if and only if $\{u^{\alpha_i}, v^{\alpha_j}\} \in E(\Gamma_2)$ for all $i, j \in V(\Sigma)$ with $\{i,j\} \in E(\Sigma)$. If such a $\Sigma$-isomorphism exists, then $\Gamma_1$ and $\Gamma_2$ are said to be \emph{$\Sigma$-isomorphic}. It is readily seen that a $\Sigma$-isomorphism from $\Gamma_1$ to $\Gamma_1$ is a $\Sigma$-automorphism of $\Gamma_1$ defined earlier. In the special case when $\Sigma = K_n$, a $K_n$-isomorphism is called an {\em $n$-fold isomorphism}, and two graphs are {\em $n$-fold isomorphic} if they are $K_n$-isomorphic. An $n$-fold isomorphism from $\Gamma_1$ to $\Gamma_1$ is called an {\em $n$-fold automorphism} of $\Gamma_1$. A $2$-fold isomorphism from $\Gamma_1$ to $\Gamma_1$ is exactly a two-fold isomorphism from $\Gamma_1$ to $\Gamma_1$ defined earlier.

\begin{lemma}
\label{pr:pr1.2}
Let $\Gamma$ be a graph with $n$ connected components each isomorphic to $\Delta$. Then there is an $n$-fold isomorphism from $\Gamma$ to $\Delta \times K_n$.
\end{lemma}

\begin{proof}
We treat $\Gamma$ as the union of $n$ vertex-disjoint copies of $\Delta$ with vertex sets $V_1, \ldots, V_n$ respectively. For each $u \in V(\Delta)$, let $u_j$ denote the corresponding vertex in $V_j$ for $1 \leq j \leq n$. Set $V(K_n) = \{1, \ldots, n\}$. For $1 \leq i \leq n$, define $\alpha_i$ to be the mapping from $V(\Gamma) = \cup_{j=1}^n V_j$ to $V(\Delta) \times V(K_n)$ such that $u_j^{\alpha_i}= (u, (i+j-1) \text{ mod } n)$ for $u \in V(\Delta)$ and $1 \le j \le n$. Every edge of $\Gamma$ is of the form $\{u_j, v_j\}$ for some $\{u, v\} \in E(\Delta)$ and $1 \leq j \leq n$. For $1 \leq k, l \leq n$ with $k \ne l$, we have $(k+j-1) \text{ mod } n \ne (l+j-1) \text{ mod } n$, and hence $\{u_j, v_j\} \in E(\Gamma)$ if and only if $\{u_j^{\alpha_k}, v_j^{\alpha_l}\} = \{(u, (k+j-1) \text{ mod } n), (v, (l+j-1) \text{ mod } n)\} \in E(\Delta \times K_n)$. Therefore, $(\alpha_1, \ldots, \alpha_n)$ is an $n$-fold isomorphism from $\Gamma$ to $\Delta \times K_n$.
\end{proof}

\begin{lemma}
Let $\Gamma$, $\Delta$ and $\Sigma$ be connected graphs.  Then the following statements hold:
\begin{enumerate}
\item
if $(\aln) \in \auts(\Gamma)$ and $(\ben) \in \auts(\Delta)$, where $n = |V(\Sigma)|$, then
$$
((\alpha_1,\beta_1), \ldots, (\alpha_n,\beta_n)) \in \auts(\Gamma\times\Delta),
$$
where each $(\alpha_i, \beta_i)$ is regarded as a permutation of $V(\Gamma \times \Delta)$ defined by $(u, x)^{(\alpha_i, \beta_i)} = (u^{\alpha_i}, x^{\beta_i})$ for $(u, x) \in V(\Gamma \times \Delta)$; in particular, if $\Gamma$ or $\Delta$ admits a nondiagonal $\Sigma$-automorphism, then so does $\Gamma\times\Delta$;
\item
if $(\alpha_1, \ldots, \alpha_m) \in \auts(\Gamma)$ and $(\ben) \in \aut_{\Delta}(\Gamma)$, where $m = |V(\Sigma)|$ and $n = |V(\Delta)|$, then
$$
(\alpha_1\beta_1, \ldots,\alpha_1\beta_n, \alpha_2\beta_1, \ldots,\alpha_2\beta_n, \ldots, \alpha_m\beta_1, \ldots, \alpha_m\beta_n) \in \aut_{\Sigma \times \Delta}(\Gamma);
$$
in particular, if $\Gamma$ admits a nondiagonal $\Sigma$-automorphism or a nondiagonal $\Delta$-automorphism, then it admits a nondiagonal $(\Sigma \times \Delta)$-automorphism.
\end{enumerate}
\end{lemma}

\begin{proof}
(a) Set $V(\Sigma) = \{1, \ldots, n\}$. Since $(\aln) \in \auts(\Gamma)$ and $(\ben) \in \auts(\Delta)$, for any $u, v \in V(\Gamma)$, $x, y \in V(\Delta)$ and $\{i, j\} \in E(\Sigma)$, we have: $\{(u, x), (v, y)\} \in E(\Gamma\times \Delta)$ $\Leftrightarrow$ $\{u, v\} \in E(\Gamma)$ and $\{x, y\} \in E(\Delta)$ $\Leftrightarrow$ $\{u^{\alpha_i}, v^{\alpha_j}\} \in E(\Gamma)$ and $\{x^{\beta_i},y^{\beta_j}\} \in E(\Delta)$ $\Leftrightarrow$ $\{(u^{\alpha_i}, x^{\beta_i}), (v^{\alpha_j}, y^{\beta_j})\} \in E(\Gamma\times \Delta)$ $\Leftrightarrow$ $\{(u, x)^{(\alpha_i, \beta_i)}, (v, y)^{(\alpha_j, \beta_j)}\} \in E(\Gamma\times \Delta)$. Therefore, $((\alpha_1,\beta_1), \ldots, (\alpha_n,\beta_n)) \in \auts(\Gamma\times\Delta)$. Clearly, if $(\aln)$ or $(\ben)$ is nondiagonal, then $((\alpha_1,\beta_1), \ldots, (\alpha_n,\beta_n))$ is nondiagonal.

(b) Set $V(\Sigma) = \{1, \ldots, m\}$ and $V(\Delta) = \{1, \ldots, n\}$. Then $\{(i, k), (j,l)\} \in E(\Sigma \times \Delta)$ if and only if $\{i,j\} \in E(\Sigma)$ and $\{k, l\} \in E(\Delta)$. Since $(\alpha_1, \ldots, \alpha_m) \in \auts(\Gamma)$ and $(\ben) \in \aut_{\Delta}(\Gamma)$, for any $u, v \in V(\Gamma)$ and $\{(i, k), (j,l)\} \in E(\Sigma \times \Delta)$, we have: $\{u, v\} \in E(\Gamma)$ $\Leftrightarrow$ $\{u^{\alpha_i},v^{\alpha_j}\} \in E(\Gamma)$ $\Leftrightarrow$ $\{(u^{\alpha_i})^{\beta_k},(v^{\alpha_j})^{\beta_l}\} = \{u^{\alpha_i \beta_k}, v^{\alpha_j \beta_l}\} \in E(\Gamma)$. Therefore, $(\alpha_1\beta_1, \ldots,\alpha_1\beta_n, \ldots, \alpha_m\beta_1, \ldots, \alpha_m\beta_n) \in \aut_{\Sigma \times \Delta}(\Gamma)$. It is easy to see that this $(\Sigma \times \Delta)$-automorphism of $\Gamma$ is nondiagonal if either $(\alpha_1, \ldots, \alpha_m)$ or $(\ben)$ is nondiagonal.
\end{proof}
}

\section{Proof of Theorem \ref{th:thr3.7}}

\label{sec:sec3}

We will use the following known results in our proof of Theorem \ref{th:thr3.7}. Recall that $\pgs$ denotes the set of elements of $\aut(\Gamma\times\Sigma)$ which leave the partition $\{V(\Gamma)\times\{i\}:i \in V(\Sigma)\}$ invariant (\cite[Definition 2.3]{qin211}).

\begin{lemma}\label{le:ler3.6}
Let $\Gamma$ and $\Sigma$ be graphs.
\begin{enumerate}
\item
\label{it:it3.6.1}
If $\autgs = \pgs$, then $(\Gamma,\Sigma)$ is unstable if and only if at least one $\Sigma$-automorphism of $\Gamma$ is nondiagonal (\cite[Lemma~2.6]{qin211}).
\item
\label{it:it3.6.2}
If both $\Gamma$ and $\Sigma$ are connected, $R$-thin and non-bipartite, then $(\Gamma,\Sigma)$ is stable if and only if $\Gamma$ and $\Sigma$ are coprime (\cite[Lemma~3.7]{qin211}).
\end{enumerate}
\end{lemma}

Denote by $\pi_{1}$ and $\pi_{2}$ the projections from $V(\Gamma\times\Sigma)$ to $V(\Gamma)$ and $V(\Sigma)$, respectively. That is, $(u,i)^{\pi_{1}} = u$ and $(u,i)^{\pi_{2}} = i$ for any $(u, i) \in V(\Gamma\times\Sigma)$. The {\em Boolean square} of $\Gamma$, denoted by ${\mathrm B}(\Gamma)$, is the graph with vertex set $V(\Gamma)$ and edge set $\{\{u,v\}:u, v\in V(\Gamma), u\neq v, N_{\Gamma}(u)\cap N_{\Gamma}(v)\neq\emptyset\}$.
In particular, if $\Gamma$ is the cycle $C_n$ with vertices labelled $1, \ldots, n$ consecutively, then ${\mathrm B}(C_n)$ is the graph with vertex set $\{1, \ldots, n\}$ such that each $i$ is adjacent to $i+2$ and $i-2$ modulo $n$. Note that ${\mathrm B}(C_n)$ is a cycle of length $n$ when $n$ is odd, and ${\mathrm B}(C_n)$ is the union of two cycles of length $n/2$ when $n$ is even. Recall that the automorphism group of $C_n$ is the dihedral group $\dn=\langle a,b\ |\ a^n=b^2=e, bab=a^{-1}\rangle$ of order $2n$. For a graph $\Gamma$ and a partition $\mathcal{B}$ of $V(\Gamma)$, the {\em quotient graph} $\Gamma_{\mathcal{B}}$ of $\Gamma$ with respect to $\mathcal{B}$ is defined to have vertex set $\mathcal{B}$ such that two blocks $B, C$ of $\mathcal{B}$ are adjacent if and only if there exists at least one edge of $\Gamma$ with one end-vertex in $B$ and the other end-vertex in $C$.

The next technical lemma will be used to prove another lemma (Lemma \ref{le:ler3.3}) which in turn will be vital to the proof of Theorem \ref{th:thr3.7}.

\begin{lemma}
\label{le:ler3.2}
Let $\Gamma$ be an $R$-thin connected graph and $n \ge 3$ an integer with $n\neq 4$. Set $V(C_n) = \{0,1,\ldots,n-1\}$. Then the following statements hold:
\begin{enumerate}
\item
\label{le:le2.1}
for any $u \in V(\Gamma)$, $\{i,j\} \in E(\bfc)$ and $\sigma \in \aut(\Gamma\times C_n)$, if $(u, i)^{\sigma\pi_{2}}\neq(u, j)^{\sigma\pi_{2}}$, then $(u,i)^{\sigma\pi_{1}}=(u,j)^{\sigma\pi_1}$;
\item
\label{le:le2.2}
for any $u \in V(\Gamma)$, $\{i,j\} \in E(\bfc)$ and $\sigma \in \aut(\Gamma\times C_n)$, if $(u, i)^{\sigma\pi_{2}}=(u, j)^{\sigma\pi_{2}}$, then $(u,i)^{\sigma\pi_{2}}=(u,k)^{\sigma\pi_{2}}$ where $k\equiv i+2r\pmod{n}$ for any $1 \le r \le (n/2)-1$.
\end{enumerate}
Moreover, if at least one of $\Gamma$ and $C_n$ is non-bipartite, then the following statements hold:
\begin{enumerate}
\item[\rm (c)]
for any $u \in V(\Gamma)$, $\{i,j\} \in E(\bfc)$ and $\sigma \in \aut(\Gamma\times C_n)$, if $(u,i)^{\sigma\pi_{2}}\neq(u,j)^{\sigma\pi_{2}}$, then $(u,k)^{\sigma\pi_{2}}\neq(u,l)^{\sigma\pi_{2}}$ for any distinct $k, l \in \{1,\ldots, n\}$;
\item[\rm (d)]
\label{le:le2.4}
for any $\sigma \in \aut(\Gamma\times C_n)$, if there exist $u \in V(\Gamma)$ and $\{i,j\} \in E(\bfc)$ such that $(u,i)^{\sigma\pi_{2}}\neq (u,j)^{\sigma\pi_{2}}$, then $(v,i)^{\sigma\pi_{2}}  \neq(v,j)^{\sigma\pi_{2}}$ for every $v \in V(\Gamma)$;
\item[\rm (e)]
if for any $u\in V(\Gamma)$ and $\sigma \in \aut(\Gamma\times C_n)$ there exists $\{i,j\} \in E(\bfc)$ such that $(u, i)^{\sigma\pi_{2}}\neq(u, j)^{\sigma\pi_{2}}$, then for any $v, w \in V(\Gamma)$ joined by an even-length walk in $\Gamma$, and any $k \in V(C_{n})$, we have $(v, k)^{\sigma\pi_{2}}=(w, k)^{\sigma\pi_{2}}$;
\item[\rm (f)]
if there exist $u \in V(\Gamma)$ and $\{i,j\} \in E(\bfc)$ such that $(u, i)^{\sigma\pi_{2}}\neq(u, j)^{\sigma\pi_{2}}$ for any $\sigma \in \aut(\Gamma\times C_n)$, then $\autgc=\pgc$.
\end{enumerate}
\end{lemma}

\begin{proof}
In this proof operations on $V(C_n)=\{0,1,\ldots,n-1\}$ are performed modulo $n$. Note that for any $\{i,j\} \in E(\bfc)$, we have $j\equiv i+2\pmod{n}$ or $j\equiv i-2\pmod{n}$. Without loss of generality we may assume $j\equiv i+2\pmod{n}$.

(\ref{le:le2.1})
Assume that $u \in V(\Gamma)$, $\{i, j\} \in E(\bfc)$ and $\sigma \in \aut(\Gamma\times C_n)$. Set $(v,k) = (u,i)^\sigma$. Taking a vertex $(w,i+1)\in N_{\Gamma\times C_n}((u,i))$, we have $(w,i+1)^{\sigma\pi_2}=k-1$ or $k+1$. Since $\{(w,i+1), (u,i+2)\}$ is an edge of $\Gamma\times C_n$, $(u, i+2)^{\sigma}$ is adjacent to $(w, i+1)^\sigma$. Hence $(u, i+2)^{\sigma\pi_2}=k-2$, $k+2$ or $k$. Thus, either $(u, i+2)^{\sigma\pi_{2}} \in \{k+2, k-2\}$ or $(u, j)^{\sigma\pi_{2}}=k$. Note that $(u, i)^{\sigma\pi_{2}}\neq(u, j)^{\sigma\pi_{2}}$ by our assumption. Consider first the case where $(u,j)^{\sigma\pi_2}\equiv k+2\pmod{n}$. We aim to prove $(u,i)^{\sigma\pi_{1}}=(u,j)^{\sigma\pi_1}$. Suppose otherwise. Then $(u, j)^{\sigma}=(w, k+2)$ for some $w \in V(\Gamma)$. We have
\begin{align*}
|\ngc((u, i)^\sigma)\cap\ngc((u, j)^\sigma)|&=\,|\ngc((v, k))\cap\ngc((w,k+2))|\\
&=\,|\nga(v)\cap\nga(w)|\cdot|\ncy(k)\cap\ncy(k+2)|\\
&=\,|\nga(v)\cap\nga(w)|.
\end{align*}
On the other hand,
\begin{align*}
|\ngc((u, i)^\sigma)\cap\ngc((u, j)^\sigma)|&=\,|\ngc((u, i))\cap\ngc((u,j+2))|\\
&=\,|\nga(u)\cap\nga(u)|\cdot|\ncy(i)\cap \ncy(j+2)|\\
&=\,|\nga(u)\cap\nga(u)|\\
&=\, |\nga(u)|.
\end{align*}
Thus, $|N_{\Gamma}(u)|=|N_{\Gamma}(v)\cap N_{\Gamma}(w)|$, but this contradicts the assumption that $\Gamma$ is $R$-thin. This contradiction shows that we have $(u,i)^{\sigma\pi_{1}}=(u,j)^{\sigma\pi_1}$ when $(u,j)^{\sigma\pi_2}\equiv k+2\pmod{n}$. Similarly, we can prove $(u,i)^{\sigma\pi_{1}}=(u,j)^{\sigma\pi_1}$ when $(u,i+2)^{\sigma\pi_2}\equiv k-2\pmod{n}$.

(\ref{le:le2.2})
Assume that $u \in V(\Gamma)$, $\{i, j\} \in E(\bfc)$ and $\sigma \in \aut(\Gamma\times C_n)$. We proceed by induction on $r$. If $r=1$, then $k=i+2$. Thus $k=j$ and we obtain the result by the assumption. Suppose that, for some $m \ge 1$, the result is true for every $r$ between $1$ and $m-1$. That is, if $(u, i)^{\sigma\pi_{2}}=(u, j)^{\sigma\pi_{2}}$, then $(u,i)^{\sigma\pi_{2}}=(u,i+2r)^{\sigma\pi_{2}}$ for $1 \le r \leq m-1$. We aim to prove that the result is also true when $r=m$. By our hypothesis, it suffices to prove $(u, i+2m-2)^{\sigma\pi_{2}} = (u,i+2m)^{\sigma\pi_{2}}$. Suppose for a contradiction that $(u, i+2m-2)^{\sigma\pi_{2}} \neq (u,i+2m)^{\sigma\pi_{2}}$. Then by part~(\ref{le:le2.1}) we have $(u, i+2m-2)^{\sigma\pi_{1}}= (u, i+2m)^{\sigma\pi_{1}}$. Since $(u,i+2m-2)^{\sigma\pi_2}\neq (u, i+2m)^{\sigma\pi_2}$, either $(u, i+2m)^{\sigma\pi_2}=(u, i+2m-2)^{\sigma\pi_2}+2$ or $(u, i+2m)^{\sigma\pi_2}=(u, i+2m-2)^{\sigma\pi_2}-2$. Consider the former case first. In this case we have $(u, i+2m-2)^{\sigma\pi_{2}}=l$ and $(u, i+2m)^{\sigma\pi_{2}}\equiv l+2\pmod{n}$ for some $l \in V(C_n)$. So $(u,i+2m-4)^{\sigma\pi_{2}}=(u,i+2m-2)^{\sigma\pi_{2}}=l$. Thus there exists a vertex $(v, i+2m-3) \in \ngc((u,i+2m-4))\cap\ngc((u, i+2m-2))$ such that $(v, i+2m-3)^{\sigma\pi_{2}}\equiv l+1\pmod{n}$ and  $(v, i+2m-3)^{\sigma} \in \ngc((u,i+2m-4)^\sigma)\cap\ngc((u, i+2m-2)^\sigma)$. So $(v, i+2m-3)^\sigma \in \ngc((u, i+2m)^{\sigma})$, but this contradicts the fact that $(v, i+2m-3)$ is not adjacent to $(u, i+2m)$. Similarly, one can derive a contradiction in the case when $(u, i+2m)^{\sigma\pi_2}=(u, i+2m-2)^{\sigma\pi_2}-2$. Therefore, we must have $(u, i+2m-2)^{\sigma\pi_{2}} = (u,i+2m)^{\sigma\pi_{2}}$ and hence the result is true when $r = m$. This completes the proof by mathematical induction.

In the rest of the proof we assume that at least one of $\Gamma$ and $C_n$ is non-bipartite. Equivalently, $n$ is odd when $\Gamma$ is bipartite.

(c) Suppose that $u \in V(\Gamma)$, $\{i, j\} \in E(\bfc)$ and $\sigma \in \aut(\Gamma\times C_n)$
such that $(u,i)^{\sigma\pi_{2}}\neq(u,j)^{\sigma\pi_{2}}$, where $j\equiv i+2\pmod{n}$. Set $(u,i)^{\sigma\pi_{2}} = i_1$ and assume without loss of generality that $(u, i+2)^{\sigma\pi_{2}}\equiv i_1+2\pmod{n}$. Since both $\Gamma$ and $C_n$ are connected and at least one of them is non-bipartite, $\Gamma\times C_n$ is connected. So for any $(u,i) \in V(\Gamma\times C_n)$, there exists a walk from $(u,i)$ to $(u, i+1)$ with vertices in $V(\Gamma)\times \{i, i+1\}$, say, $(u,i)=(u_0, i), (u_1, i+1), (u_2, i),\ldots, (u_{2r},i), (u_{2r+1}, i+1)=(u, i+1)$, for some $r$ between $0$ and $|V(\Gamma)|-1$. Since $(u,i)^{\sigma\pi_{2}}=i_1$ and $(u, i+2)^{\sigma\pi_{2}}=i_1+2\pmod{n}$, for each $v \in \nga(u)$ we have $(v, i+1)^{\sigma\pi_{2}}\equiv i_1+1\pmod{n}$ and $(v, i-1)^{\sigma\pi_{2}}\equiv i_1-1\pmod{n}$. Combining this with $u_1\in N_\Gamma(u)$, we obtain that $(u_1, i+1)^{\sigma\pi_{2}}\equiv i_1+1\pmod{n}$.
By part \eqref{le:le2.2}, we have $(u_1, i+1)\neq (u_1, i+1+2r)$ for any integer $r$ between $0$ and $n/2$. Thus for each $w \in N_{\Gamma}(u_1)$ we have $(w,i)^{\sigma\pi_{2}}=i_1$ and $(w,i+2)^{\sigma\pi_{2}}\equiv i_1+2\pmod{n}$. Hence $(u_2, i)^{\sigma\pi_{2}}=i_1$. Applying successively the same argument to pairs of every other vertices in the walk above, we obtain that $(u_{2t+1}, i+1)^{\sigma\pi_{2}}\equiv i_1+1\pmod{n}$ for $0 \le t \le r$ and $(u_{2t}, i)^{\sigma\pi_{2}}=i_1$ for $1 \le t \le r$. Thus $(u, i+1)^{\sigma\pi_{2}}\equiv i_1+1\pmod{n}$. Similarly, we can prove that for any $0\leq t\leq n$, $(u,i+t)^{\sigma\pi_{2}}\equiv i_1+t\pmod{n}$ when $(u,i)^{\sigma\pi_{2}}=i_1$ and $(u, i+2)^{\sigma\pi_{2}}\equiv i_1+2\pmod{n}$. Hence $(u,k)^{\sigma\pi_{2}}\neq(u,l)^{\sigma\pi_{2}}$ for any distinct $k, l \in \{1, \ldots, n\}$.

(d) Let $\sigma \in \aut(\Gamma\times C_n)$. Assume that there exist $u \in V(\Gamma)$ and $\{i,j\} \in E(\bfc)$ such that $(u,i)^{\sigma\pi_{2}}\neq (u,j)^{\sigma\pi_{2}}$. We aim to prove $(v,i)^{\sigma\pi_{2}}  \neq(v,j)^{\sigma\pi_{2}}$ for all $v \in V(\Gamma)$. Since this inequality holds when $v = u$, we assume $v \ne u$ in the sequel. Since $(u,i)^{\sigma\pi_{2}}\neq (u,j)^{\sigma\pi_{2}}$, by part (c), we have $(u,k)^{\sigma\pi_{2}}\neq(u,l)^{\sigma\pi_{2}}$ for any distinct $k, l \in \{1,2,\ldots, n\}$. Thus, if $u$ and $v$ are adjacent in $\Gamma$, then $(u,i-1)^{\sigma\pi_{2}}$, $(u,i+1)^{\sigma\pi_{2}}$ and $(u,i+3)^{\sigma\pi_{2}}$ are pairwise distinct, and hence $(v,i)^{\sigma\pi_{2}}\neq (v,i+2)^{\sigma\pi_{2}}$. By part (c) again, we have $(v, k)^{\sigma\pi_{2}}\neq (v,l)^{\sigma\pi_{2}}$ for any distinct $k, l \in \{ 1, 2,\ldots, n\}$, and so $(v,i)^{\sigma\pi_{2}}\neq(v,j)^{\sigma\pi_{2}}$. If $u$ and $v$ are not adjacent in $\Gamma$, then there is a path in $\Gamma$ from $u$ to $v$ with length at least $2$, say, $u=u_0, u_1,\ldots, u_p = v$. As shown above, we have $(u_1, i)^{\sigma\pi_{2}}\neq (u_1,i+2)^{\sigma\pi_{2}}$. Similarly to the proof in the case when $u$ and $v$ are adjacent, we can prove $(u_2, i)^{\sigma\pi_2}\neq (u_2,i+2)^{\sigma\pi_2}$. Applying successively what we have proved to other vertices in the path $u=u_0, u_1,\ldots, u_p=v$, we obtain $(u_k, i)^{\sigma\pi_2}\neq (u_k, i+2)^{\sigma\pi_2}$ for $0 \le k \le p$. In particular, we have $(v, i)^{\sigma\pi_2}\neq (v, i+2)^{\sigma\pi_2}$. By part (c), we then obtain $(v,i)^{\sigma\pi_{2}}\neq(v, j)^{\sigma\pi_{2}}$ as desired.

(e) Assume that for any $u\in V(\Gamma)$ and $\sigma \in \aut(\Gamma\times C_n)$ there exists $\{i,j\} \in E(\bfc)$ such that $(u, i)^{\sigma\pi_{2}}\neq(u, j)^{\sigma\pi_{2}}$. Let $v, w \in V(\Gamma)$ be such that $d(v,w)=2r$ is even. Consider the case $r=1$ first. In this case we have $\nga(v)\cap\nga(w)\neq\emptyset$ and so we may take $x \in \nga(v)\cap\nga(w)$. For any $k\in V(C_n)$, set $W=\{(x,k-1), (x,k+1)\}$. Then $W \subseteq \ngc((v, k))\cap\ngc((w,k))$. Since $\sigma \in \aut(\Gamma\times C_n)$, we have $W^{\sigma}=\{(x, k-1)^{\sigma}, (x, k+1)^{\sigma}\}\subseteq\ngc((v, k)^\sigma)\cap\ngc((w,k)^\sigma)$. Since $(u,i)^{\sigma\pi_{2}}\neq(u,j)^{\sigma\pi_{2}}$, by part (d), we have $(x, i)^{\sigma\pi_{2}}\neq (x, j)^{\sigma\pi_{2}}$ for any $x \in V(\Gamma)$. Moreover, by part (c), we have $(x,k-1)^{\sigma\pi_{2}}\neq (x, k+1)^{\sigma\pi_{2}}$. Combining this with part (a), we obtain $(x, k-1)^{\sigma\pi_{1}}=(x, k+1)^{\sigma\pi_{1}}$. Note that $(v, k)^\sigma$ and $(w,k)^\sigma$ are in $\ngc((x, k-1)^\sigma)\cap \ngc((x, k+1)^\sigma)$ and $(x, k-1)^{\sigma\pi_{1}}=(x, k+1)^{\sigma\pi_{1}}$. Hence either $(v, k)^{\sigma\pi_{2}}=(w,k)^{\sigma\pi_{2}}$ or $C_n=C_4$. Since $n\neq 4$ by our assumption, we have $(v, k)^{\sigma\pi_{2}}=(w,k)^{\sigma\pi_{2}}$ as required. Now assume $r\geq 2$. Since $d(v,w)=2r$, there is a path in $\Gamma$ from $v$ to $w$ with length $2r$. Applying what we have proved to pairs of every other vertices in this path, we can obtain the desired result.

(f) Suppose that there exist $u \in V(\Gamma)$ and $\{i,j\} \in E(\bfc)$ such that $(u, i)^{\sigma\pi_{2}}\neq(u, j)^{\sigma\pi_{2}}$ for any $\sigma \in \aut(\Gamma\times C_n)$. We may assume $j\equiv i+2\pmod{n}$ without loss of generality. By part (d), we have $(v, i)^{\sigma\pi_{2}}\neq (v, j)^{\sigma\pi_{2}}$ for any $v \in V(\Gamma)$. Moreover, by part (c), we have $(v,k)^{\sigma\pi_{2}}\neq (v, l)^{\sigma\pi_{2}}$ for any distinct $k, l \in \{ 1, 2 \cdots, n-1\}$.

Consider first the case where $n$ is odd. We aim to prove $(v, k)^{\sigma\pi_{2}}=(w, k)^{\sigma\pi_{2}}$ for any distinct $v, w \in V(\Gamma)$, any $k \in V(C_n)$, and any $\sigma \in \autgc$. If there is a walk from $v$ to $w$ in $\Gamma$ with even length, then we obtain the result by part (e). Assume there is no such a walk in the sequel. Consider the case where $v$ and $w$ are adjacent first. In this case, for any $i \in V(C_{n})$, we can construct a cycle $C_{v,w}$ of length $2n$ in $\Gamma\times C_n$ containing $(v, i)$, namely
\[
C_{v,w}: (v, i), (w,i+1), (v,i+2),(w,i+3) \cdots, (v, i-1), (w, i),(v,i+1),\cdots,(w, i-1), (v,i).
\]
Since $(v, i)^{\sigma\pi_{2}}\neq(v, i+2)^{\sigma\pi_{2}}$ for any $\sigma\in \aut(\Gamma\times C_n)$, by part (a) we have $(v,i)^{\sigma\pi_{1}}=(v, i+2)^{\sigma\pi_{1}}$. Since $C_{v,w}$ is a cycle in $\Gamma\times C_{n}$ and $(v, r)$ and $(w,r)$ are antipodal in $C_{v,w}$ for $0 \le r \le n-1$, $C_{v,w}^\sigma$ is a cycle in $\Gamma\times C_{n}$ and $(v, i)^{\sigma\pi_{2}}$ and $(w, i)^{\sigma\pi_{2}}$ are antipodal in $C_{v,w}^{\sigma}$. Therefore, $(v, i)^{\sigma\pi_{2}}=(w, i)^{\sigma\pi_{2}}$. Now consider the case where $v$ and $w$ are not adjacent. Since $\Gamma$ is connected, there exists a walk from $v$ to $w$ of odd length in $\Gamma$, say, $v=v_0, v_1,\ldots, v_{2r+1}=w$, for some $r\ge 1$. Then for any $\sigma \in \aut(\Gamma\times C_{n})$, as shown above,  we have $(v_1, k)^{\sigma\pi_{2}}=(v_0, k)^{\sigma\pi_{2}}$ for any $k \in V(C_{n})$. Since there is a walk from $v_1$ to $v_{2r+1}$ with even length, we have $(v_1, k)^{\sigma\pi_{2}}=(v_{2r+1}, k)^{\sigma\pi_{2}}$ for any $k \in V(C_{n})$. Thus, $(v_0, k)^{\sigma\pi_{2}} = (v_{2r+1}, k)^{\sigma\pi_{2}}$, that is, $(v, k)^{\sigma\pi_{2}}=(w, k)^{\sigma\pi_{2}}$, for any $k \in V(C_{n})$, which implies that $\sigma\in \pgc$ for any $\sigma \in \autgc$. Hence $\autgc=\pgc$ when $n$ is odd.

Now consider the case where $n\ge 4$ is even. Since by our assumption at least one of $\Gamma$ and $C_n$ is non-bipartite, in this case $\Gamma$ must be non-bipartite. Set $\Delta=\Gamma\times C_{n}$ and $\mathcal{ B}=\{V(\Gamma)\times\{j\}: j\in V(C_{n})\}$. Then the quotient graph $\Delta_{\mathcal{B}}$ of $\Delta$ with respect to the partition $\mathcal{B}$ of $V(\Delta)$ is isomorphic to $C_{n}$. Since $n$ is even, $C_n$ is bipartite and hence $\Gamma\times C_n$ is bipartite. Moreover, since $\Gamma\times C_n$ is connected, any two vertices $(v, i), (w, i)\in V(\Gamma\times C_n)$ are joined by a path of even length in $\Gamma\times C_n$, which implies that there is a walk of even length from $v$ to $w$ in $\Gamma$. Thus, by part (e), we obtain that any $\sigma \in \autgc$ satisfies $(v, k)^{\sigma\pi_{2}}=(w, k)^{\sigma\pi_{2}}$ for any $v, w \in V(\Gamma)$ and $k \in V(C_{n})$, and hence any $\sigma \in \autgc$ belongs to $\mathrm{P}(\Gamma, C_{n})$. Therefore, $\aut(\Gamma\times C_{n})\subseteq \mathrm{P}(\Gamma, C_{n})$ and consequently $\aut(\Gamma\times C_{n}) = \mathrm{P}(\Gamma, C_{n})$ when $n$ is even.
\end{proof}

The following lemma plays a key role in our proof of Theorem \ref{th:thr3.7}. It will also be used to prove Theorems \ref{co:co3.3} and \ref{th:th3.4} in the next section.

\begin{lemma}
\label{le:ler3.3}
Let $\Gamma$ be an $R$-thin connected graph and $n \ge 3$ an integer with $n\neq 4$. Suppose that at least one of $\Gamma$ and $C_n$ is non-bipartite. Then the following statements hold:
\begin{enumerate}
\item
\label{it:it3.3.1}
if $n$ is compatible with $\Gamma$, then $\autgc=\pgc$;
\item
\label{it:it3.3.2}
if $n \ge 5$ is odd and for every $\{u,v\} \in E(\Gamma^*)$, $\nga(u)\cap\nga(v)$ is not an independent set of $\Gamma$, then $\autgc=\pgc$.
\end{enumerate}
\end{lemma}

\begin{proof}
Set $V(C_n) = \{0, 1, \ldots, n-1\}$, in which operations are taken modulo $n$. Recall that for any $\{i,j\} \in E(\bfc)$, $j\equiv i+2\pmod{n}$ or $j\equiv i-2\pmod{n}$. Without loss of generality we may assume $j\equiv i+2\pmod{n}$ whenever $\{i, j\} \in E(\bfc)$. By Lemma \ref{le:ler3.2}(f), we only need to prove the existence of $u \in V(\Gamma)$ and $\{i,j\} \in E(\bfc)$ such that $(u,i)^{\sigma\pi_{2}}\neq(u,j)^{\sigma\pi_{2}}$ for any $\sigma \in \autgc$.

(\ref{it:it3.3.1})
We only prove the result when $n$ is even as the proof for odd $n$ is similar. Since $n$ is even and $n$ is compatible with $\Gamma$, we may take a vertex $u \in V(\Gamma)$ such that $n/2 \notin L(u)$. Suppose for a contradiction that there exist $(v,i)\in V(\Gamma\times C_n)$ and $\sigma\in \aut(\Gamma\times C_n)$ such that $(v,i)^{\sigma\pi_{2}}=(v,i+2)^{\sigma\pi_{2}}$. Then by Lemma~\ref{le:ler3.2}(d) we have $(w, i)^{\sigma\pi_{2}}=(w, i+2)^{\sigma\pi_{2}}$ for each $w\in V(\Gamma)$. Set $u = (w,i)^{\sigma\pi_{1}}$ and $m = (n/2) - 1$. By Lemma \ref{le:ler3.2}(\ref{le:le2.2}), for any $r$ between $0$ and $m$, we have $(w,i)^{\sigma\pi_{2}}=(w, k)^{\sigma\pi_{2}}$, where $k\equiv i+2r\pmod{n}$. Thus there are $n/2$ distinct vertices $u=u_0, u_1,\ldots, u_{m} \in V(\Gamma)$ and one vertex $l \in V(C_n)$ such that $(w,k)^{\sigma}=(u_r,l)$ where $k\equiv i+2r\pmod{n}$ for $0 \le r \le m$. Since $\sigma \in \aut(\Gamma\times C_n)$, it follows that $\deg((w, k))=\deg((u_r,l))$. Moreover, since $C_n$ is a regular graph with degree $2$, we obtain further that $\deg(w)=\deg(u_r)$ for $0 \le r \le m$. Note that for $0 \le r \le m$ we have
\begin{align*}
|\ngc((w, i+2r)^\sigma)\cap\ngc((w, i+2r+2)^\sigma)|&=\,|\ngc((u_r, l))\cap\ngc((u_{r+1},l))|\\
&=\,|\nga(u_r)\cap\nga(u_{r+1})|\cdot|\ncy(l)|\\
&=\,2|\nga(u_r)\cap\nga(u_{r+1})|.
\end{align*}
On the other hand,
\begin{align*}
|\ngc((w, i+2r)^\sigma)\cap\ngc((w, i+2r+2)^\sigma)|&=\,|\ngc((w, i+2r))\cap\ngc((w,i+2r+2))|\\
&=\,|\nga(w)\cap\nga(w)|\cdot|\ncy(i+2r)\cap \ncy(i+2r+2)|\\
&=\,|\nga(w)\cap\nga(w)|\\
&=\, |\nga(w)|.
\end{align*}
Thus $|N_{\Gamma}(w)|=2|N_{\Gamma}(u_r)\cap N_{\Gamma}(u_{r+1})|$ for $0 \le r \le m$. So there is a cycle in $\Gamma^*$ with length $n/2$ which contains $u$, but this is a contradiction. Therefore, we have $(w,i)^{\sigma\pi_{2}}\neq (w,j)^{\sigma\pi_{2}}$ for any $\{i, j\}\in E(\bfc)$. Since this holds for any $\sigma\in \aut(\Gamma\times C_n)$, by Lemma \ref{le:ler3.2}(e) we obtain that $(v, i)^{\sigma\pi_{2}}\neq(v, i+2)^{\sigma\pi_{2}}$ for any $(v,i) \in V(\Gamma\times C_n)$ and $\sigma \in \aut(\Gamma\times C_n)$. This together with Lemma \ref{le:ler3.2}(f) yields $\autgc=\pgc$.

(\ref{it:it3.3.2})
Suppose, for a contradiction, that there exist $(u,i) \in V(\Gamma\times C_n)$ and $\sigma \in \autgc$ such that $(u,i)^{\sigma\pi_{2}}=(u, i+2)^{\sigma\pi_{2}}$. Since $n$ is odd, by Lemma~\ref{le:ler3.2}(\ref{le:le2.2}), we have $(u,i)^{\sigma\pi_{2}}=(u,k)^{\sigma\pi_{2}}$ for any $k \in V(C_n)$. So there are vertices $(v, k), (w, k) \in V(\Gamma\times C_n)$ such that $(u, i)^{\sigma}=(v,k)$ and $(u,i+2)^\sigma=(w,k)$. Note that $\deg(u)=\deg(v)=\deg(w)$ and $\deg(u)=2|N_{\Gamma}(v)\cap N_{\Gamma}(w)|$. By our assumption, $N_{\Gamma}(v)\cap N_{\Gamma}(w)$ contains two adjacent vertices, say, $x$ and $y$. By parts (b) and (d) of Lemma \ref{le:ler3.2}, we obtain that $(v, l)^{\sigma\pi_{2}}=(v,l+2)^{\sigma\pi_{2}}$ and $(w, l)^{\sigma\pi_{2}}=(w,l+2)^{\sigma\pi_{2}}$ for any $l \in V(C_n)$. Moreover, since $n$ is odd, we have $(v, l)^{\sigma\pi_{2}}=(v,l+2r)^{\sigma\pi_{2}}$ and $(w, l)^{\sigma\pi_{2}}=(w,l+2r)^{\sigma\pi_{2}}$ for $0 \le r \le (n/2)-1$. Let $k_1 = (v,l)^{\sigma\pi_{2}}$ and $k_2 = (w, l)^{\sigma\pi_{2}}$. Since $N_{\Gamma}(v)\cap N_{\Gamma}(w) \neq \emptyset$ and $n>3$, either $k_1=k_2$ or $k_1\in \{k_2+2, k_2-2\}$.

By parts (b) and (d) of Lemma \ref{le:ler3.2} again, we also obtain that $(x, l)^{\sigma\pi_{2}}=(x,l+2r)^{\sigma\pi_{2}}$ and $(y, l)^{\sigma\pi_{2}}=(y,l+2r)^{\sigma\pi_{2}}$ for any $l \in V(C_n)$ and  $0 \le r \le (n/2)-1$. Let $k_3 = (x, l)^{\sigma\pi_{2}}$ and $k_4 = (y, l)^{\sigma\pi_{2}}$.

Since $x$ is adjacent to $y$, $(x, l)$ is adjacent to $(y, l+1)$ and $(y, l-1)$. Hence $(x, l)^{\sigma\pi_{2}}\equiv (y, l+1)^{\sigma\pi_{2}}+1\pmod{n}$ or $(x, l)^{\sigma\pi_{2}}\equiv (y, l-1)^{\sigma\pi_{2}}-1\pmod{n}$, that is, $k_3\equiv k_4-1\pmod{n}$ or $k_3\equiv k_4+1\pmod{n}$. Since $v$ is adjacent to $x$ and $y$, $(v, l)$ is adjacent to $(x, l-1)$ and $(y, l-1)$, which implies that $(v,l)^{\sigma\pi_{2}}$ is adjacent to $(x, l-1)^{\sigma\pi_{2}}$ and $(y, l-1)^{\sigma\pi_{2}}$.
Hence $(v, l)^{\sigma\pi_{2}}\equiv (x, l-1)^{\sigma\pi_{2}}+1\pmod{n}$ or $(v, l)^{\sigma\pi_{2}}\equiv (x, l-1)^{\sigma\pi_{2}}-1\pmod{n}$, that is, $k_1\equiv k_3+1\pmod{n}$ or $k_1\equiv k_3-1\pmod{n}$.
Similarly, we have $k_1\equiv k_4+1\pmod{n}$ or $k_1\equiv k_4-1\pmod{n}$.

Since $w$ is adjacent to $x$ and $y$, $(w, l)$ is adjacent to $(x, l-1)$ and $(y, l-1)$, which implies that $(w, l)^{\sigma\pi_{2}}$ is adjacent to $(x, l-1)^{\sigma\pi_{2}}$ and $(y,l-1)^{\sigma\pi_{2}}$. Thus $(w, l)^{\sigma\pi_{2}}\equiv (x, l-1)^{\sigma\pi_{2}}+1\pmod{n}$ or $(w, l)^{\sigma\pi_{2}}\equiv (x, l-1)^{\sigma\pi_{2}}-1\pmod{n}$, that is, $k_2 \equiv k_3+1\pmod{n}$ or $k_2 \equiv k_3-1\pmod{n}$.
Similarly, we have $k_2\equiv k_4+1\pmod{n}$ or $k_2\equiv k_4-1\pmod{n}$.

The two paragraphs above show that we have $k_1 \equiv k_2\pmod{n}$, $k_1 \equiv k_2 - 2\pmod{n}$, or $k_1 \equiv k_2 + 2\pmod{n}$. If $k_1\equiv k_2\pmod{n}$, then we may assume $k_1\equiv k_2\equiv k_3-1\pmod{n}$ and $k_1\equiv k_2\equiv k_4-1\pmod{n}$, which contradicts the fact that $k_3\equiv k_4+1\pmod{n}$ or $k_3\equiv k_4-1\pmod{n}$. Hence $k_1 \not \equiv k_2\pmod{n}$. Suppose that $k_1\equiv k_2 - 2\pmod{n}$. Then by $k_1\equiv k_3+1\pmod{n}$ or $k_3-1\pmod{n}$ and $k_1\equiv k_4+1\pmod{n}$ or $k_4-1\pmod{n}$, we obtain that $k_2\equiv k_3+1\pmod{n}$ or $k_3-3\pmod{n}$ and $k_2\equiv k_4+1\pmod{n}$ or $k_4-3\pmod{n}$. However, we have $k_2\equiv k_3-1\pmod{n}$ or $k_3+1\pmod{n}$ and $k_2\equiv k_4-1\pmod{n}$ or $k_4+1\pmod{n}$. Hence $k_2\equiv k_3+1\pmod{n}$ and $k_2\equiv k_4+1\pmod{n}$. Thus, $k_3\equiv k_4\pmod{n}$, which is a contradiction. Therefore, $k_1 \not \equiv k_2-2$. Similarly, we can prove that $k_1 \not \equiv k_2 + 2\pmod{n}$. This final contradiction shows that for any $(u,i) \in V(\Gamma\times C_n)$ and $\sigma \in \autgc$ we have $(u,i)^{\sigma\pi_{2}}\neq (u, i+2)^{\sigma\pi_{2}}$. Using this and Lemma \ref{le:ler3.2}, we obtain $\autgc=\pgc$ immediately.
\end{proof}

\begin{Proof}\textit{\ref{th:thr3.7}}.
Theorem \ref{th:thr3.7} follows from Lemmas \ref{le:ler3.6} and \ref{le:ler3.3} immediately. \qed
\end{Proof}

\section{Proofs of Theorems \ref{co:co3.3}, \ref{th:th3.4} and \ref{thm:coines}}
\label{sec:uaut}

In~\cite{wilson08}, Wilson introduced the concept of expected automorphisms of $\Gamma\times K_2$. We now generalize this concept from $\Gamma\times K_2$ to $\Gamma\times C_n$. Set $V(C_n) = \{0, 1, \ldots, n-1\}$. Recall that $\aut(C_n)$ is isomorphic to the dihedral group $\dn$. For any $\delta \in \dn$, let $\lbar{\delta}$ be the permutation of $V(\Gamma\times C_n)$ defined by
\[
(u,i)^{\lbar{\delta}}=(u, i^{\delta}), \text{ for } (u, i) \in V(\Gamma\times C_n).
\]
It is readily seen that $\lbar{\delta}$ is an automorphism of $\Gamma\times C_n$. Set $\lbar{\aut(C_n)} = \{\lbar{\delta}\,:\, \delta\in \dn\}$. Then $\lbar{\aut(C_n)}$ is a subgroup of $\aut(\Gamma\times C_n)$ isomorphic to $\aut(C_n)$. Similarly, for any $\sigma \in \aut(\Gamma)$, let $\lbar{\sigma}$ be the automorphism of $\Gamma\times C_n$ defined by
\[
(u,i)^{\lbar{\sigma}}=(u^{\sigma}, i), \text{ for } (u, i) \in V(\Gamma\times C_n).
\]
Set $\lbar{\aut(\Gamma)} = \{\lbar{\sigma}\,:\,\sigma\in\aut(\Gamma)\}$. Then $\lbar{\aut(\Gamma)}$ is a subgroup of $\aut(\Gamma\times C_n)$ isomorphic to $\aut(\Gamma)$. Note that $\lbar{\sigma}\ \lbar{\delta}=\lbar{\delta}\ \lbar{\sigma}$.

It is known that $\aut(\Gamma)\times \aut(C_n)$ is isomorphic to a subgroup of $\aut(\Gamma\times C_n)$. In fact, it is isomorphic to the subgroup
$$
\mathrm{R}(\Gamma, C_n) := \langle\lbar{\aut(\Gamma)},\lbar{\aut(C_n)}\rangle
$$
of $\aut(\Gamma\times C_n)$ generalized by $\lbar{\aut(\Gamma)}$ and $\lbar{\aut(C_n)}$. Note that $\mathrm{R}(\Gamma, C_n)$ is a subgroup of $\pgc$ and is the normalizer of $\lbar{\aut(C_n)}$ in $\autgc$. The elements of $\mathrm{R}(\Gamma, C_n)$ are called the {\em expected automorphisms} of $\Gamma\times C_n$, and the elements of $\mathrm{P}(\Gamma, C_n) \setminus  \mathrm{R}(\Gamma, C_n)$ are called the {\em unexpected automorphisms} of $\Gamma\times C_n$. It is readily seen that for any $\gamma \in \mathrm{R}(\Gamma, C_n)$, $u \in V(\Gamma)$, and $i, j \in V(C_n)$, we have $(u, i)^{\gamma \pi_{1}} = (u,j)^{\gamma\pi_{1}}$.
Of course, if $\autgc = \mathrm{R}(\Gamma, C_n)$, then $(\Gamma,C_n)$ is stable; otherwise, $(\Gamma,C_n)$ is unstable.

For a vertex $u$ of $\Gamma$ and a subgroup $H$ of $\aut(\Gamma)$, let $u^{H}=\{u^h:h\in H\}$ be the $H$-orbit on $V(\Gamma)$ containing $u$. Let $\Gamma/H$ be the quotient graph of $\Gamma$ with respect to the partition $\{u^{H}: u\in V(\Gamma)\}$ of $V(\Gamma)$ (\cite{wilson08}). That is, $\Gamma/H$ is the graph with vertex set $\{u^{H}: u\in V(\Gamma)\}$ in which $\{u^{H}, v^{H}\}$ is an edge if and only if $u^{H} \ne v^{H}$ and there exist $u' \in u^H$ and $v' \in v^H$ such that  $\{u', v'\} \in E(\Gamma)$. In particular, for $\gamma \in \aut(\Gamma)$, we write $\Gamma/\gamma$ in place of $\Gamma/\langle \gamma \rangle$. An \emph{arc} of $\Gamma$ is an ordered pair of adjacent vertices $(u, v)$ of $\Gamma$. Denote by $A(\Gamma)$ the set of arcs of $\Gamma$. The following concepts are all extracted from \cite{wilson08}.

\smallskip
\begin{defn}
\label{de:de3.2}
(\cite{wilson08})~Let $\Gamma$ be a connected graph.
\begin{enumerate}
\item
Let $\gamma \in \aut(\Gamma)$ be an involution and $\Gamma_1 = \Gamma/\gamma$. Let $\alpha_1 \in \aut(\Gamma_1)$ and let $\pi$ be the natural projection from $\Gamma$ to $\Gamma/\gamma$ (that is, $\pi$ sends $u \in V(\Gamma)$ to $u^{\langle\gamma\rangle} \in V(\Gamma/\gamma)$). A {\em covering permutation} of $\alpha_1$ is a permutation $\alpha$ of $V(\Gamma)$ such that $(u^\alpha)^\pi=(u^\pi)^{\alpha_1}$ for every $u \in V(\Gamma)$. An {\em anti-automorphism} of $\Gamma$ is a permutation $\alpha$ of $V(\Gamma)$ which commutes with $\gamma$ such that $\{u^\alpha, v^{\gamma\alpha}\} \in E(\Gamma)$ whenever $\{u,v\} \in E(\Gamma)$.
\item
Let $\phi$ be an automorphism of $\Gamma$ with order at least $3$. If $\{u^\phi,v^{\phi^{-1}}\} \in E(\Gamma)$ whenever $\{u,v\} \in E(\Gamma)$, then $\Gamma$ is called a {\em (generalized) cross-cover} of $\Gamma/\phi$.
\item
Let $\Gamma_1$ be a graph and $\alpha \in \aut(\Gamma_1)$. Let $H=\mathbb{Z}^n_2$ and $\gamma \in \aut(H)$. Let $L$ be a map from $V(\Gamma_1)$ to the set of subgroups of $H$ such that $L(v^\alpha)=L(v)^\gamma$ for all $v \in V(\Gamma_1)$. Let $\omega$ be a map from $A(\Gamma_1)$ to the set of subsets of $H$ such that $\omega(v, u)=\{x^{-1}: x \in \omega(u, v)\}$ for any $(u, v) \in A(\Gamma_1)$ and there is a fixed element $h_0 \in H \setminus (\cap_{v\in V}L(v))$ satisfying $\omega((u, v)^\alpha)=\omega(u, v)^\gamma+h_0$ for all $(u, v) \in A(\Gamma_1)$. Define $\Gamma=GV(\Gamma_1,H, L,\omega)$ to be the graph with vertex set $V(\Gamma)=\{(v,L(v)+h): v\in V(\Gamma_1), h\in H\}$ and arc set
\[
A(\Gamma)=\{((u, L(u)+h),(v, L(v)+h+x)): (u,v)\in A(\Gamma_1), h\in H, x \in \omega(u,v)\}.
\]
We call $\Gamma=GV(\Gamma_1,H, L,\omega)$ a {\em twist} of $\Gamma_1$.
\end{enumerate}
\end{defn}

The following lemma is extracted from \cite{lauri15} and \cite{wilson08}.

\begin{lemma}
\label{le:le3.11}
The following hold:
\begin{enumerate}
\item If a graph has an anti-automorphism, or is a cross-cover or twist of some graph, then it must be unstable (\cite[Theorems 2--4]{wilson08}).
\item If a graph is nontrivially unstable, then it has an anti-automorphism, or is a cross-cover or a twist of a graph (\cite[Theorem~5]{wilson08}).
\item If $\Gamma$ is an $R$-thin connected non-bipartite graph, then the following conditions are equivalent:
\begin{enumerate}
\item $\Gamma$ is unstable;
\item $\Gamma$ is nontrivially unstable;
\item $\Gamma$ has a nondiagonal two-fold automorphism;
\item $\Gamma$ has an anti-automorphism, or is a cross-cover or twist of some graph.
\end{enumerate}
\end{enumerate}
\end{lemma}

\begin{proof}
A nontrivially unstable graph is precisely an $R$-thin connected non-bipartite unstable graph. Thus an $R$-thin connected non-bipartite graph is nontrivially unstable if and only if it is unstable. By \cite[Theorem 3.2]{lauri15}, a graph is unstable if and only if it has a nondiagonal two-fold automorphism. Combining all these with the statements in (a) and (b), we conclude that conditions (i)--(iv) in (c) are equivalent for $R$-thin connected non-bipartite graphs.
\end{proof}

Note that $\pgs \cong \auts(\Gamma)\rtimes\aut(\Sigma)$ for any graph pair $(\Gamma,\Sigma)$ (see \cite[Lemma 2.4]{qin211}). So, for any $\sigma \in \pgs$, there exist $(\aln) \in \aut_{\Sigma}(\Gamma)$ and $\delta \in \aut(\Sigma)$ such that $\sigma$ corresponds to $((\aln), \delta) \in \auts(\Gamma)\rtimes\aut(\Sigma)$. In the case when $\Sigma = C_n$, $(\aln)$ is uniquely determined by $\sigma$ and $(u,i)^{\sigma\pi_{1}} = u^{\alpha_i}$ holds for any $(u, i) \in V(\Gamma\times C_n)$.

\delete
{
For any $i \in V(C_{2k}) = \{0, 1, \ldots, 2k-1\}$, let $\Gamma_{i}$ be the subgraph of $\Gamma\times C_{2k}$ induced by $V(\Gamma)\times\{i,i+1\}$. That is, $V(\Gamma_{i}) = V(\Gamma)\times\{i,i+1\}$ and $E(\Gamma_{i}) = \{\{(u, i), (v, i+1)\}: \{u, v\} \in E(\Gamma)\}$. For any $\sigma \in \mathrm{P}(\Gamma, C_{2k})$ and $i \in V(C_{2k})$, define $\sigma_i$ to be the automorphism of $\Gamma_{i}$ induced by $\sigma$; that is, $\sigma_i$ is defined by $(u,i)^{\sigma_{i}}=(u, i)^{\sigma}$ and $(u,i+1)^{\sigma_{i}}=(u, i+1)^{\sigma}$ for $(u, i), (u, i+1) \in V(\Gamma_i)$.

\begin{lemma}
\label{le:le3.9}
Let $\Gamma$ be an $R$-thin connected non-bipartite graph and $k \ge 3$ an integer. Set $V(C_{2k}) = \{1, \ldots, 2k\}$ and let $\sigma \in \mathrm{P}(\Gamma, C_{2k})$. If $\sigma_i$ is an expected automorphism of $\Gamma_i$ for every $i \in V(C_{2k})$, then $\sigma$ is an expected automorphism of $\Gamma\times C_{2k}$.
\end{lemma}

\begin{proof}
Suppose to the contrary that $\sigma$ is an unexpected automorphism of $\Gamma\times C_{2k}$. Then there are distinct $(u, i), (u, j) \in V(\Gamma)\times\{i, j\}$ such that $(u,i)^\sigma=(u_1,i')$ and $(u,j)^\sigma=(u_2, j')$ for some different vertices $u_1, u_2$ of $\Gamma$ and vertices $i', j'$ of $C_{2k}$. In particular, there exists $l\in V(C_{2k})$ such that $(u, l-1)^{\sigma\pi_1}\neq(u,l)^{\sigma\pi_1}$. So $\sigma_{l-1}$ is an unexpected automorphism of $\Gamma_{l-1}$, but this contradicts the assumption that $\sigma_i$ is an expected automorphism of $\Gamma_i$ for every $i\in V(C_{2k})$.
\end{proof}
}

\begin{lemma}
\label{th:th3.5}
Let $\Gamma$ be an $R$-thin connected non-bipartite graph and $k \ge 3$ an integer. If $2k$ is compatible with $\Gamma$ and $(\Gamma, C_{2k})$ is nontrivially unstable, then $\Gamma$ has an anti-automorphism, or is a cross-cover or twist of some graph.
\end{lemma}

\begin{proof}
Set $V(C_{2k}) = \{1, \ldots, 2k\}$. Since $\Gamma$ is an $R$-thin connected non-bipartite graph and $2k$ is compatible with $\Gamma$, by Lemma \ref{le:ler3.3}, we have $\autge=\pge$. Since $(\Gamma, C_{2k})$ is nontrivially unstable, there are unexpected automorphisms of $\Gamma \times C_{2k}$. Consider an arbitrary unexpected automorphism $\sigma \in \autge$. For each $i \in V(C_{2k})$, define $u^{\alpha_i} = (u,i)^{\sigma\pi_1}$ for $u \in V(\Gamma)$. Note that for any $\{u,v\}\in E(\Gamma)$ and $\{i, i+1\}\in E(C_{2k})$ we have $\{u^{\alpha_i}, v^{\alpha_j}\}=\{(u,i)^{\sigma\pi_1}, (v, i+1)^{\sigma\pi_1}\}$. Since $\{(u,i), (v, i+1)\} \in E(\Gamma\times C_{2k})$ and $\sigma \in \aut(\Gamma\times C_{2k})$, we see that $\{(u, i)^{\sigma}, (v, i+1)^{\sigma}\} \in E(\Gamma\times C_{2k})$ and so $\{u^{\alpha_i}, v^{\alpha_{i+1}}\} \in E(\Gamma)$. Hence $\{u,v\} \in E(\Gamma)$ if and only if $\{u^{\alpha_i}, v^{\alpha_j}\} \in E(\Gamma)$. So $(\alpha_1,\ldots,\alpha_{2k})$ is a $C_{2k}$-automorphism of $\Gamma$. Since $\sigma$ is an unexpected automorphism, there exist distinct vertices $(u, i), (u, j) \in V(\Gamma\times C_{2k})$ such that $(u,i)^{\sigma\pi_1} \neq (u,j)^{\sigma\pi_1}$, that is, $u^{\alpha_i} \neq u^{\alpha_j}$ and therefore $\alpha_i\neq\alpha_j$. Let $\delta$ be the unique automorphism of $C_{2k}$ satisfying $i^\delta=2k-i \pmod{2k}$ for $i\in V(C_{2k})$. By Lemma \ref{le:le1.1}, $(\alpha_1,\ldots,\alpha_{2k})^\delta = (\alpha_1, \alpha_{2k}, \alpha_{2k-1},\ldots,\alpha_{k+2}, \alpha_{k+1}, \alpha_{k},\ldots, \alpha_2)$ is a $C_{2k}$-automorphism of $\Gamma$. Since $\aut_{C_{2k}}(\Gamma)$ is a group, $(\alpha^{-1}_1,\alpha^{-1}_{2k},\alpha^{-1}_{2k-1}, \ldots, \alpha^{-1}_{k+2}, \alpha^{-1}_{k+1},\ldots,\alpha^{-1}_2)$ is a $C_{2k}$-automorphism of $\Gamma$. Hence $(\id,\alpha_2\alpha^{-1}_{2k}, \ldots, \id, \alpha_{k+2}\alpha^{-1}_{k},\ldots,\alpha_{2k}\alpha^{-1}_2)$ is also a $C_{2k}$-automorphism of $\Gamma$. Since $\Gamma$ is $R$-thin, by Lemma \ref{re:re1.1}, this $C_{2k}$-automorphism of $\Gamma$ must be diagonal. This implies that  $\alpha_2=\alpha_{2k}, \alpha_{3}=\alpha_{2k-1}, \ldots,\alpha_{k+2}=\alpha_{k}$. Thus $(\alpha_1, \ldots, \alpha_{2k}) = (\alpha_1, \alpha_2,\alpha_3,\ldots, \alpha_{k},\alpha_{k+1},\alpha_{k},\alpha_{k-1},\ldots,\alpha_3,\alpha_2)$. Let $\tau$ be the unique automorphism of $C_{2k}$ satisfying $i^\tau\equiv i+2\pmod{2k}$ for $i\in V(C_{2k})$. Then $(\alpha_1, \alpha_{2}, \alpha_{3},\ldots,\alpha_{k}, \alpha_{k+1}, \alpha_{k},\alpha_{k-1},\ldots,\alpha_3, \alpha_2)^\tau = (\alpha_3, \alpha_4, \ldots, \alpha_{k}, \alpha_{k-1}, \alpha_{k-2}, \ldots, \alpha_1, \alpha_2)$, which, by Lemma \ref{le:le1.1}, is a $C_{2k}$-automorphism of $\Gamma$. Hence $(\alpha_3^{-1}, \alpha_4^{-1}, \ldots, \alpha_{k}^{-1}, \alpha_{k-1}^{-1}, \alpha_{k-2}^{-1}, \ldots, \alpha_1^{-1}, \alpha_2^{-1})$ is a $C_{2k}$-automorphism of $\Gamma$ and so $(\alpha_1\alpha_3^{-1}, \alpha_2\alpha_4^{-1}, \ldots, \id, \alpha_{k+1}\alpha_{k-1}^{-1}, \alpha_{k}\alpha_{k-2}^{-1}, \ldots, \alpha_{3}\alpha_1^{-1}, \id)$ is also a $C_{2k}$-automorphism of $\Gamma$. Since $\Gamma$ is $R$-thin, $\alpha_1=\alpha_3=\cdots=\alpha_{2k-1}$ and $\alpha_{2}=\alpha_4=\cdots=\alpha_{2k}$. So we have proved that for any unexpected automorphism of $\Gamma \times C_{2k}$ there corresponds a $C_{2k}$-automorphism of $\Gamma$ with the form $(\alpha_1,\alpha_{2},\ldots,\alpha_1, \alpha_2)$. Note that $\alpha_1\neq\alpha_2$ and $\{u^{\alpha_1}, v^{\alpha_2}\} \in E(\Gamma)$ for any $\{u,v\}\in E(\Gamma)$. Thus $(\alpha_1, \alpha_2)$ is a two-fold automorphism of $\Gamma$. Since $\alpha_1 \neq \alpha_2$, this two-fold automorphism is nondiagonal and hence $\Gamma$ is unstable. Thus, by Lemma \ref{le:le3.11}(b), $\Gamma$ has an anti-automorphism, or is a cross-cover or twist of some graph.
\end{proof}

Now we are ready to prove Theorems \ref{co:co3.3}, \ref{th:th3.4} and \ref{thm:coines}.

\smallskip
\begin{Proof}\textit{\ref{co:co3.3}}.
(a) (``only if" part)~Set $V(C_{2k}) = \{0, 1, \ldots, 2k-1\}$. Suppose $(\Gamma, K_2)$ is unstable. Since $C_4$ is not $R$-thin, $(\Gamma, C_4)$ is unstable, and hence the result is true when $k=2$. It remains to prove that $(\Gamma,C_{2k})$ is unstable for all $k \ge 3$. Since $(\Gamma, K_2)$ is unstable, there is a nondiagonal two-fold automorphism $(\alpha_1,\alpha_2)$ of $\Gamma$. Define a permutation $\sigma$ of $V(\Gamma\times C_{2k})$ as follows: For any $(u, i) \in V(\Gamma\times C_{2k})$, if $i$ between $0$ and $2k-1$ is odd, then set $(u,i)^{\sigma} = (u^{\alpha_2}, i)$; if $i$ between $0$ and $2k-1$ is even, then set $(u,i)^{\sigma} = (u^{\alpha_1}, i)$. Since $\alpha_1$ and $\alpha_2$ are permutations of $V(\Gamma)$, $\sigma$ is indeed a permutation of $V(\Gamma\times C_{2k})$.

We now prove $\sigma\in \mathrm{P}(\Gamma, C_{2k})$. To this end we first prove $\sigma\in \aut(\Gamma\times C_{2k})$. For any $\{(u, i), (v, i+1)\} \in E(\Gamma\times C_{2k})$,
$\{(u, i)^\sigma, (v, i+1)^\sigma\}$ is equal to $\{(u, i)^{\alpha_1}, (v, i+1)^{\alpha_2}\}=\{(u^{\alpha_1}, i),(v^{\alpha_2}, i+1)\}$ or $\{(u, i)^{\alpha_2}, (v, i+1)^{\alpha_1}\} = \{(u^{\alpha_2}, i), (v^{\alpha_1}, i+1)\}$. Thus $\{(u, i), (v, i+1)\}^\sigma \in E(\Gamma\times C_{2k})$ by the definition of $(\alpha_1,\alpha_2)$. Hence $\sigma$ maps edges of $\Gamma\times C_{2k}$ to edges of $\Gamma\times C_{2k}$. Moreover, since $\sigma$ is a permutation of $V(\Gamma\times C_{2k})$, it maps different edges of $\Gamma\times C_{2k}$ to different edges of $\Gamma\times C_{2k}$. Thus, $\sigma \in \aut(\Gamma\times C_{2k})$. By the definition of $\sigma$, for any $(u, i) \in V(\Gamma\times C_{2k})$, we have $(u, i)^{\sigma\pi_{2}}=i$. Hence $\sigma \in \mathrm{P}(\Gamma, C_{2k})$. Moreover, since $(\alpha_1,\alpha_2)$ is nondiagonal, we have $\alpha_1\neq \alpha_2$ and so there exists $v \in V(\Gamma)$ such that $v^{\alpha_1}\neq v^{\alpha_2}$. This together with the definition of $\sigma$ implies that $(v,2i)^{\sigma \pi_1} \neq (v,2i+1)^{\sigma \pi_1}$ for $0\leq i\leq k$. Let $\delta$  be the unique automorphism of $C_{2k}$ satisfying $i^\delta\equiv i+1\pmod{2k}$ for any $i\in V(C_{2k})$. Then $\overline{\delta}\in \overline{\aut(C_{2k})}$ is an automorphism of $\Gamma\times C_{2k}$ satisfying $(u, i)^{\bar{\delta}}=(u,i+1)$ for any $(u, i)\in V(\Gamma\times C_{2k})$. We have $\sigma\bar{\delta}\neq\bar{\delta}\sigma$. Thus $\sigma$ is an unexpected automorphism of $\Gamma\times C_{2k}$. Hence $(\alpha_1,\alpha_2,\alpha_1,\alpha_2,\ldots,\alpha_1,\alpha_2 )$ is a nondiagonal $C_{2k}$-automorphism of $\Gamma$. Therefore, by \cite[Lemma~2.6]{qin211}, $\Gamma\times C_{2k}$ is unstable.

(b) (``only if" part)~Let $k \ge 3$ be an integer such that $2k$ is compatible with $\Gamma$. If $(\Gamma, C_{2k})$ is nontrivially unstable, then by Lemma \ref{th:th3.5}, $\Gamma$ satisfies condition (iv) in Lemma \ref{le:le3.11}(c). Hence, by Lemma \ref{le:le3.11}(c), $\Gamma$ is unstable, that is, $(\Gamma, K_2)$ is unstable.

(a) (``if" part)~Suppose $(\Gamma, C_{2k})$ is unstable for every integer $k \geq 2$. Since $\Gamma$ is a finite graph, there exists an integer $l \geq 3$ such that $2l$ is compatible with $\Gamma$. (In fact, we can take $l$ to be an integer such that $l \notin L(u)$ for every vertex $u\in V(\Gamma)$.) Since $(\Gamma, C_{2l})$ is unstable by our assumption, by the ``only if" part of (b), $(\Gamma, K_2)$ is unstable.

(b) (``if" part)~Let $k \ge 3$ be an integer. Suppose $(\Gamma, K_2)$ is unstable. Then by the ``only if" part of (a), $(\Gamma, C_{2k})$ is unstable. Since $2k$ is compatible with $\Gamma$, $\Gamma$ and $C_{2k}$ must be coprime. Note that $C_{2k}$ is $R$-thin as $2k \ge 6$ and $\Gamma$ is non-bipartite by our assumption. Hence $(\Gamma, C_{2k})$ is nontrivially unstable.

(c) Suppose $k \ge 3$ is an integer such that $2k$ is compatible with $\Gamma$ and $(\Gamma, C_{2k})$ is nontrivially unstable. Then, by (b), $(\Gamma, K_2)$ is unstable. Hence, by the ``only if" part of (a), $(\Gamma, C_{2l})$ is unstable for every $l \ge 3$. \qed
\end{Proof}

\smallskip
\begin{Proof}\textit{\ref{th:th3.4}}.
Set $V(C_{2k+1}) = \{0, 1, \ldots, 2k\}$. Since $\Gamma$ is an $R$-thin connected graph and $C_{2k+1}$ is an odd cycle, both $\Gamma$ and $C_{2k+1}$ are $R$-thin and at least one of them is non-bipartite. Since $2k+1$ is compatible with $\Gamma$, we have $\autgo=\pgo$ by Lemma \ref{le:ler3.3}(a). Suppose for a contradiction that $(\Gamma, C_{2k+1})$ is unstable. Then $(\Gamma, C_{2k+1})$ is nontrivially unstable. So there exists an unexpected automorphism $\sigma$ of $\Gamma\times C_{2k+1}$. Define $u^{\alpha_i} = (u,i)^{\sigma \pi_1}$, $u \in V(\Gamma)$, for each $i \in V(C_{2k+1})$. Note that for any $\{u, v\}\in E(\Gamma)$ and $\{i, i+1\}\in E(C_{2k+1})$, we have $\{u^{\alpha_i}, v^{\alpha_{i+1}}\}=\{(u, i)^{\sigma\pi_1}, (v,i+1)^{\sigma\pi_1}\}$. Since $\sigma \in \aut(\Gamma\times C_{2k+1})$ and $\{(u, i), (v, i+1)\} \in E(\Gamma\times C_{2k+1})$, we obtain that $\{u,v\} \in E(\Gamma)$ if and only if $\{(u,i)^{\sigma\pi_1}, (v, i+1)^{\sigma\pi_1}\} \in E(\Gamma)$. So $(\alpha_1, \alpha_2,\ldots,\alpha_{2k+1})$ is a $C_{2k+1}$-automorphism of $\Gamma$. Since $\sigma$ is an unexpected automorphism of $\Gamma\times C_{2k+1}$, there exists an automorphism $\bar{\delta}\in \overline{\aut(C_{2k+1})}$ such that $\sigma\bar{\delta}\neq\bar{\delta}\sigma$ and hence $\bar{\delta}^{-1}\sigma\bar{\delta}\neq \sigma$. Suppose that $(u, i)^{\bar{\delta}}=(u,j)$ for some $(u, i), (u, j)\in V(\Gamma\times C_{2k+1})$. Then $(u, i)^{\sigma\bar{\delta}}=((u,i)^{\sigma\pi_1},i)^{\bar{\delta}}=((u,i)^{\sigma\pi_1},j)\neq(u, i)^{\bar{\delta}\sigma}=(u,j)^\sigma$. Thus there exist distinct vertices $(u, i), (u, j) \in V(\Gamma\times C_{2k+1})$ such that $(u,i)^{\sigma\pi_1}\neq(u,j)^{\sigma\pi_1}$ and hence $\alpha_i\neq\alpha_j$. Hence $(\alpha_1,\ldots,\alpha_{2k+1})$ is nondiagonal. Let $\tau$ be the permutation of $V(C_{2k+1})$ such that $i^\tau=2k+1-i\pmod{2k+1}$ for any $i\in V(C_{2k+1})$. Then $\{i^\tau, j^\tau\}=\{2k+1-i,2k+1-j\}$ for any $\{i, j\}\in E(C_{2k+1})$. Since $\{i,j\} \in E(C_{2k+1})$, we have $j\equiv i+1\pmod{2k+1}$ or $j\equiv i-1\pmod{2k+1}$. Without loss of generality, we may assume $j\equiv i+1\pmod{2k+1}$. We can easily verify that $\{2k+1-i,2k+1-j\}=\{2k+1-i, 2k-i\}$ is also an edge of $C_{2k+1}$. Since $\tau$ is a permutation of $V(C_{2k+1})$ and maps edges of $C_{2k+1}$ to edges of $C_{2k+1}$, we see that $\tau$ is an automorphism of $C_{2k+1}$. Since $\tau$ maps $i$ to $2k+1-i\pmod{2k+1}$ for each $i\in V(C_{2k+1})$, we have $(\alpha_1,\ldots,\alpha_{2k+1})^\tau=(\alpha_{1},\alpha_{2k+1},\ldots, \alpha_{k+2}, \alpha_{k+1}, \ldots,\alpha_{2})$. Since $(\alpha_1,\ldots,\alpha_{2k+1})\in \aut_{C_{2k+1}}(\Gamma)$, it follows from Lemma \ref{le:le1.1} that $(\alpha_1,\alpha_{2k+1},\ldots, \alpha_{k+2}, \alpha_{k+1}, \ldots,\alpha_{2})$ is a $C_{2k+1}$-automorphism of $\Gamma$. Moreover, since $\aut_{C_{2k+1}}(\Gamma)$ is a group, $(\alpha^{-1}_1,\alpha^{-1}_{2k+1}, \ldots, \alpha^{-1}_{k+2}, \alpha^{-1}_{k+1},\ldots,\alpha^{-1}_2)$ is also a $C_{2k+1}$-automorphism of $\Gamma$. So $(\id,\alpha_2\alpha^{-1}_{2k+1}, \ldots, \alpha_{k+1}\alpha^{-1}_{k+2}, \alpha_{k+2}\alpha^{-1}_{k+1},\ldots,\alpha_{2k+1}\alpha^{-1}_2)$ is a $C_{2k+1}$-automorphism of $\Gamma$. Since $\Gamma$ is $R$-thin, by Lemma \ref{re:re1.1}, this $C_{2k+1}$-automorphism of $\Gamma$ must be diagonal. That is, $\alpha_i = \alpha_{2k+3-i}$ for $2 \le i \le 2k+1$. So $(\alpha_1,\ldots,\alpha_{2k+1})=(\alpha_1, \alpha_2,\ldots,\alpha_{k+1},\alpha_{k+1},\ldots,\alpha_2)$. Obviously, the permutation $\phi$ of $V(C_{2k+1})$ which maps $i$ to $i-1\pmod{2k+1}$ for each $i \in V(C_{2k+1})$ is an automorphism of $C_{2k+1}$. By Lemma \ref{le:le1.1} again, it follows that $(\alpha_1, \alpha_2,\ldots,\alpha_{k+1},\alpha_{k+1},\ldots,\alpha_2)^\phi=(\alpha_2, \alpha_3,\ldots,\alpha_{k+1}, \alpha_{k+1}, \ldots, \alpha_{2}, \alpha_1)$ is a $C_{2k+1}$-automorphism of $\Gamma$. Since $\aut_{C_{2k+1}}(\Gamma)$ is a group, $(\alpha^{-1}_2,\alpha^{-1}_{3}, \ldots,\alpha^{-1}_{k+1},\alpha^{-1}_{k},\ldots,\alpha^{-1}_1)$ is also a $C_{2k+1}$-automorphism of $\Gamma$. Hence
$$
(\alpha_1\alpha^{-1}_2, \alpha_2\alpha^{-1}_{3}, \ldots, \id,\alpha_{k+1}\alpha^{-1}_{k}, \alpha_{2}\alpha^{-1}_2)
$$
is a $C_{2k+1}$-automorphism of $\Gamma$. Again, by Lemma \ref{re:re1.1}, this $C_{2k+1}$-automorphism of $\Gamma$ must be diagonal. That is, $\alpha_1=\alpha_{2}=\cdots=\alpha_{2k+1}$ and so $(\alpha_1,\ldots,\alpha_{2k+1})$ is diagonal, a contradiction. Combining this with Lemma \ref{le:ler3.6}(a), we conclude that $(\Gamma, C_{2k+1})$ is stable. \qed
\end{Proof}

\smallskip
\begin{Proof}\textit{\ref{thm:coines}}.
Set $V(C_{2k+1})=\{0,1,\ldots,2k\}$. Consider an arbitrary $(\alpha_1, \alpha_2, \ldots, \alpha_{2k+1}) \in \aut_{C_{2k+1}}(\Gamma)$. Similarly to the proof of Theorem \ref{th:th3.4}, we can show that there exists an automorphism $\sigma$ of $C_{2k+1}$ such that $i^\sigma=2k+1-i\pmod{2k+1}$ for any $i\in V(C_{2k+1})$. Thus, by Lemma \ref{le:le1.1}, $(\alpha_1, \alpha_2, \ldots,\alpha_{2k+1})^\sigma=(\alpha_{1},\alpha_{2k+1},\ldots, \alpha_{k+2}, \alpha_{k+1}, \ldots,\alpha_{2}) \in \aut_{C_{2k+1}}(\Gamma)$. Since $\aut_{C_{2k+1}}(\Gamma)$ is a group, it follows that $(\alpha_{1},\alpha_{2k+1},\ldots, \alpha_{k+2}, \alpha_{k+1}, \ldots,\alpha_{2})^{-1} = (\alpha^{-1}_1,\alpha^{-1}_{2k+1}, \ldots, \alpha^{-1}_{k+2}, \alpha^{-1}_{k+1},\ldots,\alpha^{-1}_2) \in \aut_{C_{2k+1}}(\Gamma)$. Hence $(\alpha_1, \alpha_2, \ldots, \alpha_{2k+1})(\alpha^{-1}_1,\alpha^{-1}_{2k+1}, \ldots, \alpha^{-1}_{k+2}, \alpha^{-1}_{k+1},\ldots,\alpha^{-1}_2) \in \aut_{C_{2k+1}}(\Gamma)$. That is,
$$
(\id,\alpha_2\alpha^{-1}_{2k+1}, \ldots, \alpha_{k+1}\alpha^{-1}_{k+2}, \alpha_{k+2}\alpha^{-1}_{k+1},\ldots,\alpha_{2k+1}\alpha^{-1}_2) \in \aut_{C_{2k+1}}(\Gamma).
$$
Since $\Gamma$ is $R$-thin, by Lemma \ref{re:re1.1}, this $C_{2k+1}$-automorphism of $\Gamma$ must be diagonal. That is, $\alpha_i = \alpha_{2k+3-i}$ for $2 \le i \le 2k+1$. So $(\alpha_1,\alpha_2,\ldots,\alpha_{2k+1})=(\alpha_1, \alpha_2,\ldots,\alpha_{k+1},\alpha_{k+1},\ldots,\alpha_2)$. Let $\phi$ be the automorphism of $C_{2k+1}$ which maps $i$ to $i-1\pmod{2k+1}$ for each $i \in V(C_{2k+1})$. By Lemma \ref{le:le1.1}, $(\alpha_1, \alpha_2,\ldots,\alpha_{k+1},\alpha_{k+1},\ldots,\alpha_2)^\phi=(\alpha_2, \alpha_3,\ldots,\alpha_{k+1}, \alpha_{k+1}, \ldots, \alpha_{2}, \alpha_1) \in \aut_{C_{2k+1}}(\Gamma)$. Hence $(\alpha_1, \alpha_2,\ldots,\alpha_{k+1},\alpha_{k+1},\ldots,\alpha_2) (\alpha^{-1}_2,\alpha^{-1}_{3}, \ldots,\alpha^{-1}_{k+1},\alpha^{-1}_{k},\ldots,\alpha^{-1}_1) \in \aut_{C_{2k+1}}(\Gamma)$. That is,
$$
(\alpha_1\alpha^{-1}_2, \alpha_2\alpha^{-1}_{3}, \ldots, \id,\alpha_{k+1}\alpha^{-1}_{k}, \alpha_{2}\alpha^{-1}_1) \in \aut_{C_{2k+1}}(\Gamma).
$$
Since $\Gamma$ is $R$-thin, by Lemma \ref{re:re1.1} again, this $C_{2k+1}$-automorphism of $\Gamma$ must be diagonal, which implies that $\alpha_1=\alpha_{2}=\cdots=\alpha_{2k+1}$. So $(\alpha_1, \alpha_2, \ldots,\alpha_{2k+1})$ is diagonal. By the arbitrariness of $(\alpha_1, \alpha_2, \ldots,\alpha_{2k+1}) \in \aut_{C_{2k+1}}(\Gamma)$, we obtain that all $C_{2k+1}$-automorphisms of $\Gamma$ are diagonal. By part (b) of Theorem \ref{th:thr3.7}, we conclude that $(\Gamma, C_{2k+1})$ is stable.
\qed
\end{Proof}

\section{Remarks and conjectures}
\label{sec:con-rem}

In part (a) of Theorem \ref{th:thr3.7} we proved that for an $R$-thin connected graph $\Gamma$ and an integer $n \ge 3$ with $n \neq 4$ such that at least one of $\Gamma$ and $C_n$ is non-bipartite, if $n$ is compatible with $\Gamma$, then $(\Gamma, C_n)$ is nontrivially unstable if and only if at least one $C_n$-automorphism of $\Gamma$ is nondiagonal. In the case when $n$ is incompatible with $\Gamma$, we expect that $(\Gamma, C_n)$ should be unstable.

\begin{conj}
\label{conj:1}
Let $\Gamma$ be an $R$-thin connected graph and $n \ge 3$ an integer with $n \neq 4$ such that at least one of $\Gamma$ and $C_n$ is non-bipartite. If $n$ is incompatible with $\Gamma$, then $(\Gamma, C_n)$ is unstable.
\end{conj}

The following example gives two unstable pairs $(\Gamma, C_n)$ where $n$ is incompatible with $\Gamma$.

\begin{exam}
\label{ex:comp}
The graph $\Gamma$ on the left-hand side of Figure \ref{fig:gammastar} is an $R$-thin connected non-bipartite graph. As mentioned in Example \ref{ex:exin}, both $3$ and $6$ are incompatible with $\Gamma$. We claim that both $(\Gamma, C_3)$ and $(\Gamma, C_6)$ are unstable.

In fact, consider the $3$-cycle $C_3$ with vertex set $V(C_3) = \{0,1,2\}$. Let $\sigma$ be the permutation on $V(\Gamma\times C_3)$ defined as follows:
 \begin{align*}
 &(u_1,0)^\sigma=(u_1,1), (u_2,0)^\sigma=(u_4,2), (u_3,0)^\sigma=(u_1,0), \\
 &(u_4,0)^\sigma=(u_4,1), (u_5,0)^\sigma=(u_1,2), (u_6,0)^\sigma=(u_4,0), \\
 &(u_1,1)^\sigma=(u_3,1), (u_2,1)^\sigma=(u_6,2), (u_3,1)^\sigma=(u_3,0), \\
 &(u_4,1)^\sigma=(u_6,1), (u_5,1)^\sigma=(u_3,2), (u_6,1)^\sigma=(u_6,0), \\
 &(u_1,2)^\sigma=(u_5,1), (u_2,2)^\sigma=(u_2,2), (u_3,2)^\sigma=(u_5,0), \\
 & (u_4,2)^\sigma=(u_2,1), (u_5,2)^\sigma=(u_5,2), (u_6,2)^\sigma=(u_2,0).
 \end{align*}
It can be verified that $\sigma \in \aut(\Gamma\times C_3) \setminus \mathrm{P}(\Gamma, C_3)$. Hence $(\Gamma, C_3)$ is unstable.

Similarly, for the $6$-cycle $C_6$ with vertex set $V(C_6) = \{0,1,2,3,4,5\}$, define $\tau$ as follows:
\begin{align*}
 &(u_1,0)^\tau=(u_1,1),\quad (u_2,0)^\tau=(u_4,5),\quad (u_3,0)^\tau=(u_1,3),\quad (u_4,0)^\tau=(u_4,1),\\
 &(u_5,0)^\tau=(u_1,5),\quad (u_6,0)^\tau=(u_4,3),\quad
(u_1,1)^\tau=(u_5,4),\quad (u_2,1)^\tau=(u_2,2),\\
 &(u_3,1)^\tau=(u_5,0),\quad(u_4,1)^\tau=(u_2,4),\quad (u_5,1)^\tau=(u_5,2),\quad (u_6,1)^\tau=(u_2,0), \\
 &(u_1,2)^\tau=(u_3,1),\quad (u_2,2)^\tau=(u_6,5),\quad (u_3,2)^\tau=(u_3,3),\quad (u_4,2)^\tau=(u_6,1),\\
 &(u_5,2)^\tau=(u_3,5),\quad (u_6,2)^\tau=(u_6,3),\quad
 (u_1,3)^\tau=(u_1,4),\quad (u_2,3)^\tau=(u_4,2),\\
 &(u_3,3)^\tau=(u_1,0),\quad (u_4,3)^\tau=(u_4,4),\quad
 (u_5,3)^\tau=(u_1,2),\quad (u_6,3)^\tau=(u_4,0),\\
&(u_1,4)^\tau=(u_5,1),\quad (u_2,4)^\tau=(u_2,5),\quad
(u_3,4)^\tau=(u_5,3),\quad(u_4,4)^\tau=(u_2,1),\\
&(u_5,4)^\tau=(u_5,5),\quad (u_6,4)^\tau=(u_2,3),\quad
(u_1,5)^\tau=(u_3,4),\quad (u_2,5)^\tau=(u_6,2),\\
&(u_3,5)^\tau=(u_3,0),\quad (u_4,5)^\tau=(u_6,4),\quad
(u_5,5)^\tau=(u_3,2),\quad (u_6,5)^\tau=(u_6,0).
\end{align*}
It can be verified that $\tau \in \aut(\Gamma\times C_6) \setminus \mathrm{P}(\Gamma,C_6)$. Thus, $(\Gamma, C_6)$ is unstable.
\end{exam}

Let $\Gamma$ be an $R$-thin connected graph and $n \ge 3$ an odd integer. In view of Theorems \ref{th:th3.4} and \ref{thm:coines}, the stability of $(\Gamma, C_{n})$ is unknown only in the following two cases: (i) $n = 3$ and $3$ is incompatible with $\Gamma$; (ii) $n \ge 5$, $n$ is incompatible with $\Gamma$, and there exists an edge $\{u,v\}\in E(\Gamma^*)$ such that $N_{\Gamma}(u)\cap N_{\Gamma}(v)$ is an independent set of $\Gamma$. As a special case of Conjecture \ref{conj:1}, we expect that $(\Gamma, C_{3})$ is unstable in Case (i). We also expect that $(\Gamma, C_n)$ is unstable in Case (ii), as covered by the following conjecture.

\begin{conj}
\label{conj:2}
Let $\Gamma$ be an $R$-thin connected graph and $n \ge 3$ an odd integer. If there exists an edge $\{u,v\}\in E(\Gamma^*)$ such that $N_{\Gamma}(u)\cap N_{\Gamma}(v)$ is an independent set of $\Gamma$, then $\autgc \setminus \pgc \neq \emptyset$ and hence $(\Gamma, C_n)$ is unstable.
\end{conj}

In the case when $n \ge 4$ is even, we expect that the following should be true.

\begin{conj}
Let $\Gamma$ be an $R$-thin connected graph and $n\geq 4$ an even integer. If $n$ is incompatible with $\Gamma$, then $\autgc \setminus \pgc \neq \emptyset$ and hence $(\Gamma, C_n)$ is unstable.
\end{conj}

Let $\Gamma$ be the graph on the left-hand side of Figure \ref{fig:gammastar}. Then $(\Gamma, K_2)$ is stable. (In fact, for any $\sigma \in \aut(\Gamma\times K_2)$, $(u_i, 0)^{\sigma\pi_1}=(u_i,1)^{\sigma\pi_1}$ for $1 \le i \le 6$, where $V(K_2)=\{0,1\}$. So all automorphisms of $\Gamma\times K_2$ are expected automorphisms \cite{wilson08} and therefore $(\Gamma, K_2)$ is stable.) On the other hand, $6$ is incompatible with $\Gamma$ (Example \ref{ex:exin}) and $(\Gamma, C_6)$ is unstable (Example \ref{ex:comp}). So the ``only if" part of part (b) in Theorem \ref{co:co3.3} is not true if $2k$ is incompatible with $\Gamma$. Also, since $3$ is incompatible with $\Gamma$ (Example \ref{ex:exin}) and $(\Gamma, C_3)$ is unstable (Example \ref{ex:comp}), the result in Theorem \ref{th:th3.4} is not true if $2k+1$ is incompatible with $\Gamma$.

\vskip 0.2 cm
\noindent{\bf Acknowledgements}~~ We appreciate the referee for his/her helpful comments. Wang and Xu were supported by the National Natural Science Foundation of China (Grants No.12071194, 12401453), the China Postdoctoral Science Foundation (Grant No. 2024M751251) and the Postdoctoral Fellowship Program of CPSF (Grant No. GZC20240626). Zhou was supported by ARC Discovery Project DP250104965.
\vskip 0.2 cm

\noindent{\bf Data Availability Statement}~~
Data sharing is not applicable to this article, as no data sets were generated or
analyzed during the preparation of this manuscript.

\vskip 0.2 cm
\noindent{\bf Conflict of Interest}~~
The authors declare that they have no conflicts of interest.

\end{document}